\documentclass[10pt]{amsart}
\usepackage{amscd,amssymb,hyperref}
\usepackage[PostScript=dvips]{diagrams}

\newtheorem{thm}{Theorem}[section]
\newtheorem{lem}[thm]{Lemma}
\newtheorem{cor}[thm]{Corollary}
\newtheorem{prop}[thm]{Proposition}

\newtheorem{rem}[thm]{Remark}
\newtheorem{example}[thm]{Example}

\def\square{\vbox{
      \hrule height 0.4pt
      \hbox{\vrule width 0.4pt height 5.5pt \kern 5.5pt \vrule width 0.4pt}
      \hrule height 0.4pt}}

\def\colim{\operatorname{co l i m}}
\def\id{\mathrm{id}}

\def\Ker{\mathrm{K er}}
\def\ch\mathrm{c h}

\def\Hom{\mathrm{H o m}}

\def\End{\mathrm{End}}
\def\dim{\mathrm{dim\,}}
\def\End{\mathrm{End}}

\def\Coker{\mathrm{C o k er}}

\def\res{\mathrm{res}}

\def\Mod{\mathrm{Mod}}

\def\Lie{\mathrm{Lie}}

\def\coalg{\mathrm{co a l g}}
\def\Tor{\mathrm{T or}}

\def\GL{\mathrm{G L}}

\newcommand{\Z}{\mathbb{Z}}

\newcommand{\bfk}{\ensuremath{{\mathbf{k}}}}

\let\la=\langle
\let\ra=\rangle

\numberwithin{equation}{section}

\begin{document}

\newcommand{\auths}[1]{\textrm{#1},}
\newcommand{\artTitle}[1]{\textsl{#1},}
\newcommand{\jTitle}[1]{\textrm{#1}}
\newcommand{\Vol}[1]{\textbf{#1}}
\newcommand{\Year}[1]{\textrm{(#1)}}
\newcommand{\Pages}[1]{\textrm{#1}}

\author{J.  Li$^*$}

\author{F.  Lei}

\author{J. Wu$^{\dag}$}

\address{ Institute of Mathematics and Physics, Shijiazhuang Railway Institute, Shijiazhuang 050043, China}\email{yanjinglee@163.com}

\address{School of Mathematical Sciences, Dalian University of Technology, Dalian 116024, China}\email{fclei@dlut.edu.cn}

\address{Department of Mathematics\\
National University of Singapore\\
Block S17\\
10, Lower Kent Ridge Road\\
Singapore 119076 }\email{matwuj@nus.edu.sg}
\urladdr{http://www.math.nus.edu.sg/\~{}matwujie}

\begin{abstract}
In this article, we investigate the functors from modules to modules that occur as the summands of tensor powers and the functors from modules to Hopf algebras that occur as natural coalgebra summands of tensor algebras. The main results provide some explicit natural coalgebra summands of tensor algebras. As a consequence, we obtain some decompositions of Lie powers over the general linear groups.
\end{abstract}
\thanks{$^*$Partiall supported by the Academic Research Fund
from Shijiazhuang Railway Institute}
\thanks{$^{\dag}$ Research is supported in part by the Academic Research Fund of the
National University of Singapore R-146-000-101-112.}


\title[Natural Coalgebra-Split Sub Hopf Algebras of Tensor Algebras]{Module Structure on Lie Powers and Natural Coalgebra-Split Sub Hopf Algebras of Tensor Algebras}
\maketitle


\section{introduction}

The algebraic question for functorial coalgebra decompositions of the tensor algebras was arising from homotopy theory described as follows. The relevant topological question is how to decompose the loop spaces. This is a classical problem in homotopy theory with applications to homotopy groups. For instance, the classical results of Cohen-Moore-Neisendorfer~\cite{CMN} on the exponents of the homotopy groups of the spheres and Moore spaces were obtained from the study of the decompositions of the loop spaces of Moore spaces. The decompositions of the loop space functor $\Omega$ from $p$-local simply connected co-$H$-spaces to spaces was then introduced in~\cite{SW1} with the subsequent development in~\cite{STW,STW2,SW2, Theriault}. Namely one gets the natural decompositions $\Omega X\simeq \bar A(X)\times \bar B(X)$ for some homotopy functors $\bar A$ and $\bar B$ on $p$-local simply-connected co-$H$-spaces~$X$. Such decompositions may lose some information for an individual space $X$ in the sense that the functor $\bar A$ may be indecomposable but the space $\bar A(X)$ may have further decompositions. However the functorial decompositions has a good property that one can
freely change the co-$H$-spaces $X$ in the decomposition formulas because they are functorial.
Also there are examples of spaces $X$ such as Hopf invariant on complexes given in~\cite{GSW} with the property that the space $\bar A(X)$ is indecomposable for certain functor $\bar A$. A fundamental question concerning the functorial decompositions is how to determine the homology of the factors $\bar A(X)$ and $\bar B(X)$, which can be reduced to be a pure algebraic question as follows:

Let $V$ be any module over a field $\bfk=\Z/p$ and let $T(V)$ be the tensor algebra on ~$V$. Then $T(V)$ becomes a Hopf algebra by saying $V$ to be primitive. By forgetting the algebra structure on $T(V)$, we have the functor $T$ from modules to coalgebras. Let
\begin{equation}\label{equation1.1}
T(V)\cong A(V)\otimes B(V)
\end{equation}
be any natural coalgebra decomposition of $T(V)$ for some functors $A$ and $B$ from (ungraded) modules to coalgebras. From~\cite[Theorem 1.3]{SW1}, the functors $A$ and $B$ can be canonically extended as functors from graded modules to graded coalgebras and the above decomposition formula holds for any graded module $V$. Then from~\cite[Theorem 1.1]{STW2}, the functors $A$ and $B$ induce functor $\bar A$ and $\bar B$ from co-$H$-spaces to spaces with the natural decomposition $\Omega X\simeq \bar A (X)\times \bar B(X)$ with the property that there exist filtrations on mod $p$ homology $H_*(\bar A(X))$ and $H_*(\bar B(X))$ such that its associated graded modules $E^0H_*(\bar A(X))\cong A(\Sigma^{-1}\bar H_*(X))$ and $E^0H_*(\bar B(X))\cong B(\Sigma^{-1}\bar H_*(X))$, where $\Sigma^{-1}$ is the desuspension of the graded modules. Briefly speaking, any coalgebra decomposition of the functor $T$ as in formula~(\ref{equation1.1}) induces a natural decomposition of the loops on $p$-local simply connected co-$H$-spaces in which the homology of its factors can be determined by their corresponding algebraic functors.

The functors $A$ and $B$ in decomposition~(\ref{equation1.1}) are complementary to each other and so it suffices to understand one of them as a coalgebra summand of the functor $T$. There exist some important coalgebra summands of $T$ in~\cite[Theorem 6.5]{SW1} that give the functorial version of the Poincar\'e-Birkhoff-Witt Theorem. One of such functors is the functor $A^{\min}$, which is the smallest natural coalgebra summand of $T(V)$ containing $V$. The coalgebra complement of the functor $A^{\min}$, denoted by $B^{\max}$, has the property that $B^{\max}(V)$ can be chosen as a sub Hopf algebra of $T(V)$. However the determination of $A^{\min}(V)$ and $B^{\max}(V)$ seems out of current technology. As a consequence, the homology of the geometric realization $\bar A^{\min}(X)$ and $\bar B^{\max}(X)$ remains unknown. It is then important to find coalgebra summands $B$ of $T$ with the explicit information on $B(V)$ because in such a case the homology of the geometric realization $\bar B(X)$ can be understood.

The purpose of this article is to provide some explicit coalgebra summands $B$ of $T$. We are interested in the special cases where $B$ can be chosen as a sub Hopf algebra of $T$. This will give a relatively large coalgebra summand of $T$ because any subalgebra of tensor algebra is a tensor algebra, and so the complement functor $A$ of $B$ as in decomposition~(\ref{equation1.1}) becomes relatively small. Let $L_n(V)$ be the $n\,$th free Lie power on $V$, namely $L_n(V)$ is the homogenous component of the free Lie algebra $L(V)$ on $V$ of tensor length $n$. Our main result is as follows:

\begin{thm}\label{theorem1.1}
Let the ground ring be a field of characteristic $p$. Let $\{m_i\}_{i\in I}$ be finite or infinite set of positive integers prime to $p$ with each $m_i>1$. Then the sub Hopf algebra of $T(V)$
generated by
$$
L_{m_ip^r}(V)\quad \textrm{ for } i\in I,\ r\geq 0
$$
is a natural coalgebra summand of $T(V)$. In particular, the sub Hopf algebra $B(V)$ of
$T(V)$ generated by
$$
L_n(V)\quad \textrm{ for } n \textrm{ NOT a power of } p
$$
is a natural coalgebra summand of $T(V)$.
\end{thm}

By the maximum property of the functor $B^{\max}$, the sub Hopf algebras in the theorem are all contained in $B^{\max}$. According to~\cite[Proposition 11.1]{SW1} as well as~\cite{Bryant-Stohr}, the indecomposable elements in $B^{\max}$ does not has tensor length $p$ for $p>2$ and so  the sub Hopf algebra $B(V)$ coincides with $B^{\max}(V)$ up to tensor length $p^2-1$. For the case $p=2$, the sub Hopf algebra $B(V)$ coincides with $B^{\max}(V)$ up to tensor length $7$ according to the computations in~\cite{SW4}. Our sub Hopf algebra $B(V)$ is strictly smaller than $B^{\max}(V)$ for general module $V$. However if the functors $B$ and $B^{\max}$ are extended to the functors from graded modules to graded modules in the sense of~\cite{SW1}, the sub Hopf algebra $B(V)$ coincides with $B^{\max}(V)$ for graded modules $V$ of dimension $\leq p-1$ with $V_{\mathrm{even}}=0$ according to ~\cite[Theorem 1.1]{Wu5}.

There is a canonical connection between coalgebra decompositions of $T$ and the decompositions of the Lie powers $L_n(V)$ as modules over the general linear groups by restricting decomposition~(\ref{equation1.1}) to the primitives. The decompositions of Lie powers have been actively studied in the recent development of representation theory~\cite{Bryant-Johnson, Bryant-Schocker, Bryant-Stohr, Donkin-Erdmann, Erdmann-Schocker}.

The article is organized as follows. In section~\ref{section2}, we investigate the sub-quotient functors of the tensor power functors $T_n\colon V\mapsto T_n(V)=V^{\otimes n}$ from modules to modules. These special functors are of course closely related to the tensor representation of the symmetric groups and the finite dimensional polynomial representations of the general linear groups (by evaluating on a fixed module $V$). They are also related the modules over the Schur algebras and the modules over the Steenrod algebra~\cite{Kuhn1, Kuhn2,Kuhn3}.  In this section, we introduce exact functors $\gamma_n(-)$ from the category of functors from modules to modules to the category of the modules over the symmetric groups in this section which are variations of the classical Schur functor~\cite{ABW, Schur, Green}. In geometry, the summands of the tensor power functors $T_n$ are closely related to decompositions of self-smash products~\cite{SW3, Wu}.

In section~\ref{section3}, we investigate the subfunctors of the Lie powers $L_n$ that occur as the summands of the functor $T_n$, which we call $T_n$-projective subfunctors of $L_n$. According to Theorem~\ref{theorem3.9}, these functors are closely related to the summands of the Lie powers $L_n(V)$ that occur as summands of $V^{\otimes n}$ studied in~\cite{Bryant-Schocker, Donkin-Erdmann, Erdmann-Schocker}.

We give a coalgebra decomposition of $T$ called the Block Decomposition in section~\ref{section4}. According to Theorem~\ref{theorem4.3}, this is a coalgebra decomposition of $T$ in the form
$$
T\cong \bigotimes_{i=1}^\infty C^{m_i},
$$
where $\{m_i\}$ is the set of all positive integers prime to $p$ and the primitives of $C^{m_i}$ are exactly given by the primitives of the tensor algebra $T$ with tensor length $m_ip^r$ for $r\geq 0$. In other words, the primitives of the tensor algebra $T$ with tensor length $mp^r$ for $r\geq 0$ for each $m$ prime to $p$ can be blocked into a coalgebra summand $C^m$ of $T$.

The proof of Theorem~\ref{theorem1.1} is given in section~\ref{section5}. In section~\ref{section6}, we give some applications of our decomposition theorem to Lie powers by restricting to the primitives.

\section{The structure on the Tensor Powers}\label{section2}
\subsection{Tensor Algebras}\label{subsection2.1}
Let $V$ be any module and let
$$T(V)=\bigoplus_{n=0}^{\infty} V^{\otimes n}$$ be the \textit{tensor algebra} generated by $V$, where $V^{\otimes n}=\bfk$ and the multiplication on $T(V)$ is given by the formal tensor product of monomials. The tensor algebra admits the universal property that, for any associated algebra $A$ and any linear map $f\colon V\to A$, there exists a unique algebra map $\tilde f\colon T(V)\to A$ such that $\tilde f|_V=f$. In particular, the linear map
$$
V\longrightarrow T(V)\otimes T(V) \quad x\mapsto x\otimes 1+1\otimes x
$$
extends uniquely to an algebra map $\psi\colon T(V)\to T(V)\otimes T(V)$ and so $T(V)$ has the canonical Hopf algebra structure with the multiplication given by the formal tensor product of monomials and the comultiplication given by $\psi$. We refer to~\cite{MM} as a classical reference for Hopf algebras and quasi-Hopf algebras. The comultiplication $\psi$ is coassociative and cocommutative. The module $T(V)$ with the comultiplication $\psi\colon T(V)\to T(V)\otimes T(V)$ is called \textit{shuffle coalgebra} as its graded dual
$$
T^*(V)=\bigoplus_{n=0}^\infty \left(V^{\otimes n}\right)^*\cong \bigoplus_{n=0}^\infty (V^*)^{\otimes n}
$$
is the usual shuffle algebra under the multiplication $\psi^*\colon T^*(V)\otimes T^*(V)\to T^*(V)$.

We are interested in the functor $T\colon V\mapsto T(V)$. There are three terminologies on this functor:
\begin{enumerate}
\item The functor $T^H\colon V\mapsto T(V)$ from modules to Hopf algebras;
\item The functor $T^C\colon V\mapsto T(V)$ from modules  to coalgebras by forgetting the multiplication;
\item The functor $T^M\colon V\mapsto T(V)$ from modules to modules by forgetting both multiplication and comultiplication.
\end{enumerate}
The notation of the functor $T$ refers one of $T^H$, $T^C$ or $T^M$ if the working category is clear.

By taking the tensor length from the definition of $T(V)$, the functor $T^M$ admits a natural decomposition that
\begin{equation}\label{equation2.1}
T^M\cong \bigoplus\limits_{n=0}^\infty T_n,
\end{equation}
 where $T_n(V)=V^{\otimes n}$ with $T_0(V)=\bfk$. Thus the functors $T^H$, $T^C$ and $T^M$ are graded functors. From the well-known property (see for instance~\cite[Lemma 3.8]{GW}) that
 \begin{equation}\label{equation2.2}
 \Hom(T_n,T_m)=\left\{
 \begin{array}{lcl}
 0&\textrm{ if }& n\not=m\\
 \bfk(\Sigma_n)&\textrm{ if }&n=m,\\
 \end{array}\right.
 \end{equation}
 the decomposition of $T^M$ is in fact an orthogonal decomposition. A direct consequence is that the comultiplication $\psi$ (as a natural transformation) is uniquely determined by the multiplication on $T(V)$ for having the Hopf structure.

\begin{prop}
Let $\Delta_V\colon T^M(V)\to T^M(V)\otimes T^M(V)$ be a natural transformation such that $T^M(V)$ with the usual multiplication together with the comultiplication given by $\Delta_V$ is a quasi-Hopf algebra for every $V$. Then $\Delta_V=\psi_V$ for all $V$.
\end{prop}
\begin{proof}
For every $V$, from the property that $\Hom(T_n,T_m)=0$ for $n\not=m$, we have
$$
\Delta_V(T_1(V))\subseteq T_1(V)\otimes T_0(V)\oplus T_0(V)\otimes T_1(V)
$$
and the counit
$$
\epsilon_V \colon T(V)=\bigoplus_{n=0}^\infty T_n(V)\longrightarrow T_0(V)
$$
is the canonical projection by sending each $T_i(V)$ to $0$ for $i>0$ and $\epsilon|_{T_{0}(V)}=\id$. Since both $\Delta_V$ and $\psi_V$ has the counit uniquely given by $\epsilon_V$, we have
$$
\Delta_V|_{T_1(V)}=\psi_V|_{T_1(V)}.
$$
It follows that $\Delta_V=\psi_V$ because both of them are algebra map with respect to formal tensor product.
\end{proof}

\begin{rem}
Given a module $V$, of course one could have many comultiplication on $T(V)$ such that $T(V)$ is Hopf. The proposition states that $\psi$ is the only one possible comultiplication on $T(V)$ as a natural transformation.
\end{rem}

\subsection{Subfunctors of the Tensor Algebra Functor}\label{subsection2.2}
Let $\mathcal{C}$ and $\mathcal{D}$ be categories and let $A,B\colon \mathcal{C}\to \mathcal{D}$ be functors. We call $A$ is a \textit{subfunctor} (\textit{quotient functor}) of $B$ if there is a natural transformation $\phi\colon A\to B$ such that
$$
\phi_X\colon A(X)\to B(X)
$$
is injective (surjective) for every object $X\in\mathcal{C}$. A subfunctor (quotient functor) of $T$ refers to a subfunctor (quotient functor) of $T^H$, $T^C$ or $T^M$. A subfunctor (quotient functor) of $T^H$ is called a \textit{sub Hopf functor} (\textit{quotient Hopf functor}) of $T$. Similarly we have \textit{sub coalgebra functor} (\textit{quotient coalgebra functor}) of $T$ and \textit{submodule functor} (\textit{quotient module functor}) of $T$. A \textit{graded subfunctor} (\textit{graded quotient functor}) of $T$ refers to a subfunctor of $T^H$, $T^C$ or $T^M$ as functors from modules to graded Hopf algebras, graded coalgebras or graded modules, respectively.

\begin{prop}\label{proposition2.3}
Let $B$ be a subfunctor of $T$. Then $B$ is a graded subfunctor of ~$T$.
\end{prop}
\begin{proof}
Let $\phi_V\colon B(V)\to T(V)$ be the natural monomorphism and let $$B_n(V)=\phi_V^{-1}(T_n(V))$$ for any $V$. From the fact that $T=\bigoplus\limits_{n=0}^\infty T_n$ is an orthogonal decomposition, the functor $B$ admits a graded decomposition $$B=\bigoplus\limits_{n=0}^\infty B_n.$$
Thus if $B$ is a subfunctor of $T^M$, then $B$ is graded subfunctor of $T^M$. By the orthogonal property of $T$ as in equation~(\ref{equation2.2}),
 $$
 \Hom(B_n,B_m)=0
 $$
 for $n\not=m$. Thus if $B$ is a subfunctor of $T^H$ or $T^C$, then the multiplication and comultiplication on $B$ whence it is defined as a natural transformation must be graded and hence the result.
\end{proof}

\begin{cor}\label{corollary2.4}
Let $C$ be a quotient functor of $T$. Then $C$ is a graded quotient functor of $T$.
\end{cor}
\begin{proof}
Let $B$ be the kernel of $T^M\twoheadrightarrow C$ by forgetting the possible Hopf or coalgebra structure on $C$. By Proposition~\ref{proposition2.3}, $B$ is a graded subfunctor of $T^M$ and so $C$ is a graded quotient functor of $T^M$, where $C_n=T_n/B_n$ for each $n\geq 0$. By the orthogonal property of $T$ as in equation~(\ref{equation2.2}),
 $$
 \Hom(C_n,C_m)=0
 $$
 for $n\not=m$. Thus if $C$ is a quotient functor of $T^H$ or $T^C$, then the multiplication and comultiplication on $C$ whence it is defined as a natural transformation must be graded and hence the result.
\end{proof}

\subsection{The Associated Symmetric Group Modules of the Functors}\label{subsection2.3}

Let $\bar V_n$ be the $n$-dimensional $\bfk$-module with a fixed choice of basis $\{x_1,\ldots,x_n\}$. For each $1\leq i\leq n$, define the linear transformation
$$
d_i\colon \bar V_n\longrightarrow \bar V_{n-1}
$$
by setting
$$
d_i(x_j)=\left\{
\begin{array}{lcl}
x_j&\textrm{ if }& j<i,\\
0&\textrm{ if }&j=i,\\
x_{j-1}&\textrm{ if }& j>i.\\
\end{array}\right.
$$
The right $\bfk(\Sigma_n)$-action on $\bar V_n$ is given by $$x_i\cdot \sigma=x_{\sigma(i)}$$ for $1\leq i\leq n$ and $\sigma\in \Sigma_n$.
Then, for each $1\leq i\leq n$ and any $\sigma\in\Sigma_n$, clearly there exists a unique permutation $d_i\sigma\in \Sigma_{n-1}$ such that the diagram
\begin{equation}\label{equation2.3}
\begin{diagram}
\bar V_n&\rTo^{\sigma}_{x_j\mapsto x_{\sigma(j)}}&\bar V_n\\
\dTo>{d_i}&&\dTo>{d_{\sigma(i)}}\\
\bar V_{n-1}&\rTo^{d_i\sigma}&\bar V_{n-1}\\
\end{diagram}
\end{equation}
commutes. Let $B$ be a functor from modules to modules. Then $B(\bar V_n)$ is a right $\bfk(\Sigma_n)$-module induced by the action of $\bfk(\Sigma_n)$ on $\bar V_n$. Define
\begin{equation}\label{equation2.4}
\gamma_n(B)=\bigcap_{i=1}^n\Ker (B(d_i)\colon B(\bar V_n)\to B(\bar V_{n-1})).
\end{equation}
By applying the functor $B$ to diagram~(\ref{equation2.3}), $\gamma_n(B)$ is a $\bfk(\Sigma_n)$-submodule of $B(\bar V_n)$. Let
$$
\phi\colon B\to C
$$
be a natural transformation of functors from  modules to modules. Then clearly $\phi_{\bar V_n}$ induces a $\bfk(\Sigma_n)$-map
$$
\gamma_n(\phi)\colon \gamma_n(B)\longrightarrow \gamma_n(C).
$$
\begin{prop}\label{proposition2.5}
Let
$$
A\rInto^{j} B\rOnto^{p} C
$$
be a short exact sequence of functors from modules to modules. Then there is a short exact sequence of $\bfk(\Sigma_n)$-modules
$$
\gamma_n(A)\rInto^{\gamma_n(j)}\gamma_n(B)\rOnto^{\gamma_n(p)}\gamma_n(C).
$$
Thus $\gamma_n(-)$ is an exact functor from the category of functors from modules to modules to the category of $\bfk(\Sigma_n)$-modules.
\end{prop}
\textbf{Note.} The exact functor $\gamma_n(-)$ is a variation of the Schur functor given in~\cite{Green} in the following sense: Let $B$ is a sub-quotient functor of $T_n$ and let $V$ be a module with $m=\dim V\geq n$. Then $B(V)$ is a sub-quotient $\bfk(GL_m(\bfk))$-module of $V^{\otimes n}$. Let $\bar V_n$ embed into $V$ in the canonical way such that $V=\bar V\oplus V'$. In our definition, $\gamma_n(B)\subseteq B(\bar V)\subseteq B(V)$. According to~\cite[Section 1.2, p.71]{Donkin-Erdmann}, $B(V)\mapsto \gamma_n(B)$ is the Schur functor.

\begin{proof}
Define coface operation $d^i\colon \bar V_{n-1}\to \bar V_n$ by setting $$
d^i(x_j)=\left\{
\begin{array}{lcl}
x_j&\textrm{ if } & j<i\\
x_{j+1}&\textrm{ if }& j\geq i\\
\end{array}\right.
$$
for $1\leq i\leq n$. Then the sequence of module $\{\bar V_{n+1}\}_{n\geq 0}$ with faces $$d_1,\ldots,d_n\colon \bar V_{n}\to \bar V_{n-1}$$ relabeled as $d_0,\ldots,d_{n-1}$ and cofaces $$d^1,\ldots,d^n\colon \bar V_{n-1}\to \bar V_n$$ relabeled as $d^0,\ldots,d^{n-1}$ by shifting indices down by $1$ forms a bi-$\Delta$-groups in the sense of~\cite[Section 1.2]{Wu2}. By applying the functors to the bi-$\Delta$-group $\{\bar V_{n+1}\}_{n\geq 0}$, one gets a short exact sequence of bi-$\Delta$-groups
$$
\{A(\bar V_{n+1})\}_{n\geq 0}\rInto \{B(\bar V_{n+1})\}_{n\geq 0}\rOnto \{C(\bar V_{n+1})\}_{n\geq 0}.
$$
The assertion then follows by~\cite[Proposition 1.2.10]{Wu2}.
\end{proof}

\begin{cor}\label{corollary2.6}
Let $\phi\colon A\to B$ be a natural transformation between functors from modules to modules. Suppose that
$$
\gamma_n(\phi)\colon \gamma_n(A)\to \gamma_n(B)
$$
is an isomorphism for each $n\geq1$. Then
$$
\phi_V\colon A(V)\to B(V)
$$
is an isomorphism for any finite dimensional module $V$. Thus if both $A$ and $B$ preserve colimits, then $\phi$ is a natural equivalence.
\end{cor}
\begin{proof}
Let $C$ be the cokernel of $\phi$. Suppose that $C(V)\not=0$ for some finite dimensional module $V$. Let
$$
n=\min\{ k\ | \ C(V)\not=0\ \mathrm{dim}(V)=k\}.
$$
Then $\gamma_n(C)=C(\bar V_n)\not=0$. By Proposition~\ref{proposition2.5}, $\gamma_n(C)=0$ which is a contradiction. Thus $C(V)=0$ for any finite dimensional module $V$. Similarly, for $D$ the kernel of $\phi$, we have $D(V)=0$ for any finite dimensional module $V$, finishing the proof.
\end{proof}

\subsection{$T_n$-projective functors}\label{subsection2.4}

Consider the functor $T_n$. Let $\gamma_n=\gamma_n(T_n)$. Then $\gamma_n$ is the $\bfk$-submodule of $\bar V_n^{\otimes n}$ spanned by the monomials
$$
x_{\sigma(1)}\otimes\cdots \otimes x_{\sigma(n)}
$$
for $\sigma\in\Sigma_n$ with the right symmetric group action explicitly given by
\begin{equation}\label{equation2.5}
(x_{i_1}\otimes \cdots\otimes x_{i_n})\cdot \sigma=x_{\sigma(i_1)}\otimes \cdots\otimes x_{\sigma(i_n)}
\end{equation}
for $\sigma\in \Sigma_n$ and the monomials $x_{i_1}\cdots x_{i_n}\in \gamma_n$. Observe that
\begin{equation}\label{equation2.6}
\gamma_n\cong \bfk(\Sigma_n)
\end{equation}
as a $\bfk(\Sigma_n)$-module.

Let $V$ be any $\bfk$-module and let $a_1,\ldots,a_n\in V$. We write $a_1\cdots a_n$ for the tensor product $a_1\otimes\cdots\otimes a_n\in V^{\otimes n}$ if there are no confusions. Let the symmetric group $\Sigma_n$ act on $V^{\otimes n}$ by permuting positions. More precisely the left $\bfk(\Sigma_n)$-action  on $V^{\otimes n}$ is given
\begin{equation}\label{equation2.7}
\sigma\cdot (a_1\cdots a_n)=a_{\sigma(1)}\cdots a_{\sigma(n)}
\end{equation}
for $\sigma\in \Sigma_n$ and the monomials $a_1\cdots a_n\in V^{\otimes n}$. Let $B$ be any functor from modules to modules. Define the functor $\gamma_n^B(-)$ by setting
\begin{equation}\label{equation2.8}
\gamma_n^B(V)=\gamma_n(B)\otimes_{\bfk(\Sigma_n)}V^{\otimes n}
\end{equation}
for any module $V$. Clearly
$$
\gamma_n^{T_n}(V)=\gamma_n\otimes_{\bfk(\Sigma_n)}V^{\otimes n}\cong T_n(V).
$$

\begin{prop}\label{proposition2.7}
Let
$$B\rInto^{j} T_n\rOnto^p C $$
be a short exact sequence of functors from modules to modules. Then the natural isomorphism $\gamma_n\otimes_{\bfk(\Sigma_n)}V^{\otimes n}\cong T_n(V)$ induces a natural commutative diagram of exact sequences
\begin{diagram}
\gamma_n(B)\otimes_{\bfk(\Sigma_n)} V^{\otimes n} &\rTo^{\gamma_n(j)\otimes\id} &\gamma_n\otimes_{\bfk(\Sigma_n)} V^{\otimes n}&\rOnto^{\gamma_n(p)\otimes\id} &\gamma_n(C)\otimes_{\bfk(\Sigma_n)}V^{\otimes n}\\
\dTo&&\dTo>{\cong}&&\dTo\\
B(V)&\rInto& T_n(V)&\rOnto& C(V)\\
\end{diagram}
with a natural exact sequence
$$
\mathrm{Tor}^{\bfk(\Sigma_n)}_1(\gamma_n(C), V^{\otimes n})\hookrightarrow \gamma_n(B)\otimes_{\bfk(\Sigma_n)}V^{\otimes n} \rightarrow B(V)\rightarrow \gamma_n(C)\otimes_{\bfk(\Sigma_n)}V^{\otimes n}\twoheadrightarrow C(V).
$$
for any module $V$.
\end{prop}
\begin{proof}
By taking the image of $j_V$, we may consider $B(V)$ is a submodule of $T_n(V)$ for any module $V$.  Let
$$
\theta\colon \gamma_n\otimes_{\bfk(\Sigma_n)} V^{\otimes n}\to T_n(V)
$$
be the isomorphism and let $\Phi^B_V=\theta\circ (\gamma_n(j)\otimes \id).$ Observe that the isomorphism $\theta$ is given by
$$
\theta(x_1\cdots x_n\otimes a_1\cdots a_n)=a_1\cdots a_n
$$
for $a_1,\ldots,a_n\in V$. Let $a=a_1\cdots a_n\in V^{\otimes n}$ be any monomial with $a_j\in V$ for $1\leq j\leq n$ and let
$$
\alpha=\sum_{\sigma\in\Sigma_n}k_{\sigma}x_{\sigma(1)}\cdots x_{\sigma(n)}\in \gamma_n(B).
$$
Then
$$
\begin{array}{rcl}
\Phi^B_V(\alpha\otimes a_1\cdots a_n)&=&\sum\limits_{\sigma\in\Sigma_n} k_{\sigma} x_{\sigma(1)}\cdots x_{\sigma(n)}\otimes a_1\cdots a_n\\
&=&\sum\limits_{\sigma\in\Sigma_n} k_{\sigma} (x_{1}\cdots x_{n})\cdot\sigma\otimes a_1\cdots a_n\\
&=&\sum\limits_{\sigma\in\Sigma_n} k_{\sigma} x_{1}\cdots x_{n}\otimes \sigma\cdot(a_1\cdots a_n)\\
&=&\sum\limits_{\sigma\in\Sigma_n}k_{\sigma} x_{1}\cdots x_{n}\otimes a_{\sigma(1)}\cdots a_{\sigma(n)} \\
&=&\sum\limits_{\sigma\in\Sigma_n} k_{\sigma} a_{\sigma(1)}\cdots a_{\sigma(n)}\in V^{\otimes n}.\\
\end{array}
$$
Define a linear transformation $f_a\colon \bar V_n\to V$ by setting
$$
f_a(x_i)=a_i
$$
for $1\leq i\leq n$. Consider $T_n(f_a)=f_a^{\otimes n}\colon T_n(\bar V_n)\to T_n(V)$. Then
$$
T_n(f_a)(\alpha)=f_a^{\otimes n}\left(\sum_{\sigma\in\Sigma_n}k_{\sigma}x_{\sigma(1)}\cdots x_{\sigma(n)}\right)=\Phi^B_V(\alpha\otimes a_1\cdots a_n).
$$
From the commutative diagram
\begin{diagram}
B(\bar V_n)&\rInto^{j_{\bar V_n}}& T_n(\bar V_n)\\
\dTo>{B(f_a)}&&\dTo>{T_n(f_a)}\\
B(V)&\rInto^{j_V}& T_n(V),\\
\end{diagram}
since
$$
\alpha\in \gamma_n(B)\subseteq B(\bar V_n),
$$
we have
\begin{equation}\label{equation2.9}
\Phi^B_V(\alpha\otimes a_1\cdots a_n)=T_n(f_a)(\alpha)\in B(V).
\end{equation}
It follows that
$$
\mathrm{Im}(\Phi^B_V)\subseteq B(V)
$$
for any module $V$. Thus we have the left square commutative diagram in the statement.

By Proposition~\ref{proposition2.5}, there is a short exact sequence of $\bfk(\Sigma_n)$-modules
$$
\gamma_n(B)\rInto^{\gamma_n(j)} \gamma_n\rOnto^{\gamma_n(p)} \gamma_n(C).
$$
Since $\gamma_n$ is a free $\bfk(\Sigma_n)$-module, there is a exact sequence
$$
\mathrm{Tor}_1^{\bfk(\Sigma_n)}(\gamma_n(C),V^{\otimes n})\hookrightarrow
\gamma_n(B)\otimes_{\bfk(\Sigma_n)}V^{\otimes n} \to \gamma_n\otimes_{\bfk(\Sigma_n)}V^{\otimes n} \twoheadrightarrow \gamma_n(C)\otimes_{\bfk(\Sigma_n)} V^{\otimes n}.
$$
Hence we have the right square of the commutative diagram as well as the exact sequence in the statement.
\end{proof}

\begin{example}\label{example2.8}
{\rm We give an example that the natural transformation
$$
\gamma_n(B)\otimes_{\bfk(\Sigma_n)}V^{\otimes n} \longrightarrow B(V)
$$
could be neither epimorphism nor monomorphism for subfunctors $B$ of $T_n$. Let $\bfk$ be a field of characteristic $2$. The Lie power
$L_2(V)$ is the submodule of $V^{\otimes 2}$ spanned by $[a,b]=ab-ba$ for $a,b\in V$. The restricted Lie power $L^{\mathrm{res}}_2(V)$ is the submodule of $V^{\otimes 2}$ spanned by $a^2, [a,b]$ for $a,b\in V$. Then
$$
\gamma_2(L_2)=\gamma_2(L^{\mathrm{res}}_2)
$$
is the $1$-dimensional submodule of $\bfk(\Sigma_2)$ generated by $1-\tau$. Since $\bfk$ is of characteristic $2$, $\gamma_2(L_2)=\gamma_2(L^{\mathrm{res}}_2)$ is the trivial $\bfk(\Sigma_2)$-module and so
$$
\gamma_2(L^{\res}_2)\otimes_{\bfk(\Sigma_2)}V^{\otimes 2}=S_2(V),
$$
the $2$-fold symmetric product of $V$. The natural transformation
$$
\gamma_2(L^{\res}_2)\otimes_{\bfk(\Sigma_2)}V^{\otimes 2}\longrightarrow L^{\res}_2(V)
$$
is not epimorphism for $V$ with $\dim V\geq 1$ because its image is given by $L_2(V)$. The kernel of this natural transformation is measured by
$$
\mathrm{Tor}_1^{\bfk(\Sigma_2)}(\gamma_2/\Lie(2), V^{\otimes 2})\not=0
$$
for $V$ with $\dim V\geq 1$.
\hfill $\Box$

}\end{example}

Let $B$ be a functor from modules to modules. The \textit{dual functor $B^*$} is defined as follows. For any finite dimensional module $V$, define
$$
B^*(V)=B(V^*)^*,
$$
where $V^*=\Hom_{\bfk}(V,\bfk)$ is the dual $\bfk$-module of $V$, and, for a general module $V$, let
$$
B^*(V)=\colim_{V_{\alpha}}B^*(V_{\alpha})
$$
be the direct limit of the module $B^*(V_{\alpha})$ subject to the direct system given by the diagram of all finite dimensional submodules of $V$ with inclusions. Clearly $T_n^*=T_n$.

\begin{prop}\label{proposition2.9}
Let $B$ be a functor from modules to modules. Then $\gamma_n(B^*)$ is the dual $\bfk(\Sigma_n)$-module of $\gamma_n(B)$ for each $n\geq 1$.
\end{prop}
\begin{proof}
For the basis $\{x_1,\ldots,x_n\}$ for $\bar V_n$, let $\{x_1^*,\ldots,x_n^*\}$ be the standard dual basis of $\bar V_n^*$. Let
$$
\theta_n\colon \bar V_n\to \bar V^*_n
$$
be the linear transformation such that $\theta_n(x_j)=x_j^*$ for $1\leq j\leq n$. Then it is routine to check that the composite
$$
\gamma_n(B^*)\subseteq B^*(\bar V_n)=B(\bar V_n^*)^*\rTo^{B(\theta_n)^*} B(\bar V_n)^*\rOnto \gamma_n(B)^*
$$
is an isomorphism of $\bfk(\Sigma_n)$-modules.
\end{proof}

We call a direct sum of copies of $T_n$ a \textit{free $T_n$-functor}. A functor $B$ from modules to modules is called \textit{$T_n$-projective} if there exists a free $T_n$-functor $F$ together with natural transformations $s\colon B\to F$ and $r\colon F\to B$ such that $r\circ s\colon B\to B$ is a natural equivalence. In other words, a $T_n$-projective functor means a summand (or retract) of a free $T_n$-functor.

\begin{prop}\label{proposition2.10}
Let $B$ be a functor from modules to modules.
\begin{enumerate}
\item[1)] If $B$ is a $T_n$-projective functor, then $\gamma_n(B)$ is a projective $\bfk(\Sigma_n)$-module and there is a natural isomorphism
$$
\gamma_n(B)\otimes_{\bfk(\Sigma_n)}V^{\otimes n}\cong B(V)
$$
for any module $V$.
\item[2)] If $B$ is a subfunctor of a direct sum of finite copies of $T_n$ with the property that $\gamma_n(B)$ is a projective $\bfk(\Sigma_n)$-module, then $B$ is a $T_n$-projective functor. Moreover there is a natural equivalence $B\cong B^*$.
\item[3)] If $B$ is a quotient functor of a direct sum of finite copies of $T_n$ with the property that $\gamma_n(B)$ is a projective $\bfk(\Sigma_n)$-module, then $B$ is a $T_n$-projective functor. Moreover there is a natural equivalence $B\cong B^*$.
\item[4)] Let $B$ be a sub-quotient functor of a direct sum of finite copies of $T_n$. Suppose that $B$ is a $T_n$-projective functor. Then $B$ is both projective and injective in the category of sub-quotient functors of direct sums of finite copies of $T_n$.
\end{enumerate}
\end{prop}
\begin{proof}
The proof of assertion~(1) is straightforward. Assertion (3) follows from (2) by considering the dual functor.

(2). Let $B$ be a subfunctor of $F$, where $F$ is a finite direct sum of copies of $T_n$. Let $C=F/B$. Then there is a short exact sequence
$$
\gamma_n(B)\rInto \gamma_n(F)\rOnto \gamma_n(C).
$$
Since $\gamma_n(B)$ is a finitely generated projective $\bfk(\Sigma_n)$-module, it is an injective $\bfk(\Sigma_n)$-module and so the above short exact sequence splits off as $\bfk(\Sigma_n)$-modules. It follows that $\gamma_n(C)$ is a projective $\bfk(\Sigma_n)$-module because $\gamma_n(F)$ is a free $\bfk(\Sigma_n)$-module. $$
\mathrm{Tor}^{\bfk(\Sigma_n)}_1(\gamma_n(C),V^{\otimes n})=0.
$$
Since $F$ is a direct sum of copies of the functor $T_n$, we can apply the exact sequence in Proposition~\ref{proposition2.7}. In particular, the natural transformation
\begin{equation}\label{equation2.10}
\Phi^B_V\colon \gamma_n(B)\otimes_{\bfk(\Sigma_n)}V^{\otimes n}\longrightarrow B(V)
\end{equation}
is a natural monomorphism. The natural inclusion $B\hookrightarrow F$ induces an natural epimorphism $F=F^*\twoheadrightarrow B^*$. By Proposition~\ref{proposition2.7}, there is a natural epimorphism
\begin{equation}\label{equation2.11}
\Phi^{B^*}_V\colon \gamma_n(B^*)\otimes_{\bfk(\Sigma_n)}V^{\otimes n} \longrightarrow B^*(V).
\end{equation}
By Proposition~\ref{proposition2.9}, $$\gamma_n(B^*)\cong\gamma_n(B)^*\cong \gamma_n(B)$$ as $\bfk(\Sigma_n)$-modules because $\gamma_n(B)$ is a finitely generated $\bfk(\Sigma_n)$-projective modules. Let $V$ be any finite dimensional $\bfk$-module. From equations~(\ref{equation2.10}) and~(\ref{equation2.11}, we have
$$
\dim B(V)=\dim B(V^*)^*=\dim B^*(V)\leq \dim (\gamma_n(B)\otimes_{\bfk(\Sigma_n)}V^{\otimes n})\leq \dim B(V).
$$
Thus $\Phi^B_V$ and $\Phi^{B^*}_V$ are isomorphisms for any finite dimensional module $V$. Since the functors $\gamma_n(B)\otimes_{\bfk(\Sigma_n)}(-)^{\otimes n}$, $B$ and $B^*$ preserve colimits, the natural transformations $\Phi^B$ and $\Phi^{B^*}$ are natural equivalences. The assertion now follows from the fact that $\gamma_n(B)\otimes_{\bfk(\Sigma_n)}V^{\otimes n}$ is a natural summand of
$
\gamma_n(F)\otimes_{\bfk(\Sigma_n)}V^{\otimes n}\cong F(V).
$

(4). Let $\mathcal{C}$ be the category of sub-quotient functors of direct sums of copies of $T_n$. It suffices to show that $T_n$ is projective and injective in the $\mathcal{C}$. Let $B$ be an object in $\mathcal{C}$ with a natural epimorphism
$q\colon B \twoheadrightarrow T_n$. It induces an epimorphism $\gamma_n(q)\colon \gamma_n(B)\to \gamma_n(T_n)=\gamma_n$. Since $\gamma_n$ is $\bfk(\Sigma_n)$-projective, there is a $\bfk(\Sigma_n)$-cross-section $s\colon \gamma_n(T_n)\to \gamma_n(B)$. Now the natural transformation
$$
T_n(V)\cong \gamma_n(T_n)\otimes_{\bfk(\Sigma_n)}V^{\otimes n}\rTo^{s\otimes\id} \gamma_n(B)\otimes_{\bfk(\Sigma_n)}V^{\otimes n}\rTo^{\Phi^B_V} B(V)
$$
is a cross-section to $q$ and so $T_n$ is projective in $\mathcal{C}$. Since $T_n\cong T_n^*$ is self-dual, $T_n$ is also injective in $\mathcal{C}$. The proof is finished.
\end{proof}

We remark that if $B$ is a sub-quotient functor of $T_n$ with the property that $\gamma_n(B)$ is $\bfk(\Sigma_n)$-projective, it is possible that $B$ is not $T_n$-projective. For instance, for the ground field $\bfk$ being of characteristic $2$, the functor $B=L^{\res}_2/L_2$ has the property that $\gamma_2(B)=0$ with $B$ not $T_2$-projective.

\section{The Structure on Lie Power Functors}\label{section3}

\subsection{The Lie Power Functors and the Symmetric Group Modules $\Lie(n)$}\label{subsection3.1}
In this section, the ground ring is a field $\bfk$. Let $V$ be a module. The \textit{free Lie algebra $L(V)$} generated by $V$ is the smallest sub Lie algebra of the tensor algebra $T(V)$ containing $V$, where the Lie structure on $T(V)$ is given by $[a,b]=ab-ba$. The functor $L$ admits a graded structure that
$$
L(V)=\bigoplus_{n=1}^\infty L_n(V),
$$
where $L_n(V)=L(V)\cap T_n(V)$ which is called the \textit{$n$th Lie power} of $V$. By applying equation~(\ref{equation2.4}) to the functor $L_n$, we have the symmetric group module
$$
\Lie(n)=\gamma_n(L_n).
$$
Let $\bar V$ be the $n$-dimensional module with a basis $\{x_1,\ldots,x_n\}$ as in subsection~\ref{subsection2.3}. From the definition,
$$
\Lie(n)=L_n(\bar V_n)\cap\gamma_n
$$
spanned by the homogenous Lie elements of length $n$ in which each $x_i$ occurs exactly once. By
the Witt formula, $\Lie(n)$ is of dimension $(n-1)!$. Following from the antisymmetry and the Jacobi identity, $\Lie(n)$
has a basis given by the elements
$$
[[x_1,x_{\sigma(2)}],x_{\sigma(3)}],\ldots,x_{\sigma(n)}]
$$
for $\sigma\in \Sigma_{n-1}$. (See ~\cite{Cohen2}.)

\begin{prop}\label{proposition3.1}
There is a natural short exact sequence
$$
\mathrm{Tor}^{\bfk(\Sigma_n)}_1(\gamma_n/\Lie(n), V^{\otimes n}) \rInto \Lie(n)\otimes_{\bfk(\Sigma_n)}V^{\otimes n}\rOnto L_n(V)
$$
for any module $V$.
\end{prop}
\begin{proof}
By Proposition~\ref{proposition2.7}, it suffices to show that the natural transformation
$$
\Phi^{L_n}_V\colon \Lie(n)\otimes_{\bfk(\Sigma_n)}V^{\otimes n} \longrightarrow L_n(V)
$$
is an epimorphism. Let $[[a_1,a_2],\ldots a_n]\in L_n(V)$ with $a_1,\ldots,a_n\in V$. Let
$$
\alpha=[[x_1,x_2],\ldots,x_n]\in \Lie(n).
$$
Then exists unique $k_{\sigma}\in\bfk$ such that
$$
\alpha=[[x_1,x_2],\ldots,x_n]=\sum_{\sigma\in \Sigma_n}k_{\sigma}x_{\sigma(1)}\cdots x_{\sigma(n)}.
$$
Following the lines in the proof of Proposition~\ref{proposition2.7}, we have
\begin{equation}\label{equation3.1}
\Phi^{L_n}_V(\alpha\otimes a_1\cdots a_n)=\sum_{\sigma\in \Sigma_n}k_{\sigma}a_{\sigma(1)}\cdots a_{\sigma(n)}=[[a_1,a_2],\ldots,a_n].
\end{equation}
The assertion follows from the fact that $L_n(V)$ is the $\bfk$-module spanned by the Lie elements $[[a_1,a_2],\ldots,a_n]$ with $a_j\in V$.
\end{proof}

For any natural transformation $\phi\colon L_n\to L_n$, we have the $\bfk(\Sigma_n)$-linear map
$$
\gamma_n(\phi)\colon \gamma_n(L_n)=\Lie(n)\longrightarrow \gamma_n(L_n)=\Lie(n).
$$
This defines a ring homomorphism
$
\gamma\colon \End(L_n)\longrightarrow \End_{\bfk(\Sigma_n)}(\Lie(n)).
$

\begin{prop}\label{proposition3.2}
If $n\not=m$, then $\Hom(L_n,L_m)=0$. Moreover the ring homomorphism $$
\gamma\colon \End(L_n)\rTo \End_{\bfk(\Sigma_n)}(\Lie(n))
$$
is an isomorphism with a natural commutative diagram of functors
\begin{diagram}
\Lie(n)\otimes_{\bfk(\Sigma_n)} V^{\otimes n}&\rTo^{\gamma_n(\delta)\otimes\id}& \Lie(n)\otimes_{\bfk(\Sigma_n)}V^{\otimes n}\\
\dOnto>{\Phi^{L_n}_V}&&\dOnto>{\Phi^{L_n}_V}\\
L_n(V)&\rTo^{\delta_V}& L_n(V)\\
\end{diagram}
for any natural transformation $\delta\colon L_n\to L_n$.
\end{prop}
\begin{proof}
Let $\phi\colon L_n\to L_m$ be a natural transformation. Let
$\bar\beta_n\colon T_n\to L_n$ be the natural epimorphism defined by
$
\bar\beta_n(a_1\cdots a_n)=[[a_1,a_2],\ldots,a_n]
$
for any module $W$ and any monomial $a_1\cdots a_n\in T_n(W)=W^{\otimes n}$. Then the
composite $$T_n\rOnto^{\bar \beta} L_n\rTo^\phi L_m\rInto T_m$$ is a natural
transformation, which is zero as $\Hom(T_n,T_m)=0$ for $n\not=m$. Thus $\phi=0$.

For the second statement, let $\delta\colon L_n\to L_n$ be a natural transformation. Let $V$ be any module. Consider
$
[[a_1,a_2],\ldots,a_n]\in L_n(V)
$
with $a_j\in V$. Let $f_a\colon \bar V_n\to V$ be the linear map with $f_a(x_j)=a_j$ for $1\leq j\leq n$. Then there is a commutative diagram
\begin{equation}\label{equation3.2}
\begin{diagram}
\gamma_n(L_n)&\subseteq& L_n(\bar V_n)&\rTo^{L_n(f_a)}& L_n(V)\\
\dTo>{\gamma_n(\delta)}&&\dTo>{\delta_{\bar V_n}}&&\dTo>{\delta_V}\\
\gamma_n(L_n)&\subseteq& L_n(\bar V_n)&\rTo^{L_n(f_a)}& L_n(V).\\
\end{diagram}
\end{equation}
Thus
$$
\begin{array}{rlr}
&\delta_V\circ\Phi_V^{L_n}([[x_1,x_2],\ldots,x_n]\otimes a_1\cdots a_n) & \\
=&\delta_V([[a_1,a_2],\ldots,a_n]) &\textrm{ by equation~(\ref{equation3.1})}\\
=&L_n(f_a)(\gamma_n(\delta)([[x_1,x_2],\ldots,x_n]))&\textrm{ by equation~(\ref{equation3.2})}\\
=&\Phi_V^{L_n}(\gamma_n(\delta)([[x_1,x_2],\ldots,x_n])\otimes a_1\cdots a_n)&\textrm{ by equation~(\ref{equation2.9})}\\
\end{array}
$$
and so the diagram in the statement commutes. It follows that the map $$\gamma\colon \End(L_n)\to \End_{\bfk(\Sigma_n)}(\Lie(n))$$ is a monomorphism.

For showing that $\gamma$ is an epimorphism, let $\theta\colon \Lie(n)\to \Lie(n)$ be any $\bfk(\Sigma_n)$-linear map. Since $\bfk(\Sigma_n)$ is a Fr\"{o}benius algebra, the free $\bfk(\Sigma_n)$-module $\gamma_n$ is injective and so there is a commutative diagram of exact sequences of $\bfk(\Sigma_n)$-modules
\begin{diagram}
\Lie(n)&\rInto& \gamma_n&\rOnto&\gamma_n/\Lie(n)\\
\dTo>{\theta}&&\dTo>{\tilde\theta}&&\dTo>{\bar\theta}\\
\Lie(n)&\rInto& \gamma_n&\rOnto&\gamma_n/\Lie(n).\\
\end{diagram}
It follows that there is a commutative diagram of short exact sequence of functors
\begin{diagram}
\Tor^{\bfk(\Sigma_n)}_1(\gamma_n/\Lie(n),V^{\otimes n})&\rInto& \Lie(n)\otimes_{\bfk(\Sigma_n)}V^{\otimes n}&\rOnto^{\Phi^{L_n}_V}& L_n(V)\\
\dTo>{\Tor(\bar\theta,\id)}&&\dTo>{\theta\otimes \id}&&\dTo>{\delta}\\
\Tor^{\bfk(\Sigma_n)}_1(\gamma_n/\Lie(n),V^{\otimes n})&\rInto& \Lie(n)\otimes_{\bfk(\Sigma_n)}V^{\otimes n}&\rOnto^{\Phi^{L_n}_V}& L_n(V)\\
\end{diagram}
for some natural transformation $\delta\colon L_n\to L_n$. By taking $V=\bar V_n$ and restricting to the submodule $\gamma_n\subseteq \bar V^{\otimes n}$, we have the commutative diagram
\begin{diagram}
\Lie(n)&\rTo^{g}_\cong&\Lie(n)\otimes_{\bfk(\Sigma_n)}\gamma_n&\rTo^{\Phi^{L_n}_{\bar V_n}|}_{\cong}&\Lie(n)\\
\dTo>{\theta}&&\dTo>{\theta\otimes\id}&&\dTo>{\gamma_n(\delta)}\\
\Lie(n)&\rTo^{g}_\cong&\Lie(n)\otimes_{\bfk(\Sigma_n)}\gamma_n&\rTo^{\Phi^{L_n}_{\bar V_n}|}_{\cong}&\Lie(n),\\
\end{diagram}
where $g(\alpha)=\alpha\otimes x_1\cdots x_n$. Thus $\theta=\gamma_n(\delta)$
because
$$
\begin{array}{rcl}
\Phi^{L_n}_{\bar V}\circ g([[x_{\sigma(1)},x_{\sigma(2)}],\ldots,x_{\sigma(n)}])&=& \Phi^{L_n}_{\bar V}([[x_{1},x_{2}],\ldots,x_{n}]\cdot\sigma\otimes x_1\cdots x_n)\\
&=& \Phi^{L_n}_{\bar V}([[x_1,x_2],\ldots,x_n]\otimes \sigma\cdot(x_1\cdots x_n))\\
&=& \Phi^{L_n}_{\bar V}([[x_1,x_2],\ldots,x_n]\otimes x_{\sigma(1)}\cdots x_{\sigma(n)})\\
&=&[[x_{\sigma(1)},x_{\sigma(2)}],\ldots,x_{\sigma(n)}].\\
\end{array}
$$
The proof is finished.
\end{proof}

\begin{cor}\label{corollary3.3}
There is a one-to-one correspondence, multiplicity preserving, between the decompositions of
the functor $L_n$ and the decompositions of $\Lie(n)$ over
$\bfk(\Sigma_n)$.\hfill $\Box$
\end{cor}

\subsection{The $T_n$-projective Subfunctors of $L_n$}\label{subsection3.2}

Let $Q$ be a subfunctor of $L_n$. Then $Q$ is a subfunctor of $T_n$ because $L_n$ is a subfunctor of $T_n$. By Proposition~\ref{proposition2.10}, the functor $Q$ is $T_n$-projective if and only if $\gamma_n(Q)$ is a $\bfk(\Sigma_n)$-projective module. From Corollary~\ref{corollary3.3}, we have the following:

\begin{prop}\label{proposition3.4}
There is a one-to-one correspondence, multiplicity preserving, between $T_n$-projective sub
functors of $L_n$ and $\bfk(\Sigma_n)$-projective submodules of
$\Lie(n)$ given by $Q\mapsto \gamma_n(Q)$.\hfill $\Box$
\end{prop}

According to~\cite[Lemma 6.2 and Theorem 7.4]{SW1}, there exists a subfunctor $L^{\max}_n$ of $L_n$ with
$\Lie^{\max}(n)=\gamma_n(L_n^{\max})$ that has the following maximum property:
\begin{enumerate}
\item[1)] $\Lie^{\max}(n)$ is a $\bfk(\Sigma_n)$-projective submodule of $\Lie(n)$ and
\item[2)] any $\bfk(\Sigma_n)$-projective submodule of $\Lie(n)$ is isomorphic to a summand of $\Lie^{\max}(n)$ as a $\bfk(\Sigma_n)$-module.
\end{enumerate}
From the above maximum property, $\Lie^{\max}(n)$ is unique up to isomorphisms of $\bfk(\Sigma_n)$-modules. By the above proposition, $L^{\max}_n$ is unique up to natural equivalences with the maximum property that
\begin{enumerate}
\item[1)] $L^{\max}_n$ is a $T_n$-projective subfunctor of $L_n$ and
\item[2)] any $T_n$-projective subfunctor of $L_n$ is isomorphic to a summand of $L^{\max}_n$.
\end{enumerate}

From Proposition~\ref{proposition2.10}, we have the following:

\begin{prop}\label{proposition3.5}
There is a natural isomorphism
$$
\Phi^{L^{\max}_n}_V\colon \Lie^{\max}(n)\otimes_{\bfk(\Sigma_n)}V^{\otimes n}\longrightarrow L^{\max}_n(V)
$$
for any module $V$.\hfill $\Box$
\end{prop}

\subsection{The $\bfk(\GL(V))$-module $L_n(V)$}\label{subsection3.3}
In this subsection, the ground field $\bfk$ is an infinite field of characteristic $p>0$ and $V$ is a fixed $\bfk$-module with the action of the general linear group $\GL(V)=\GL_m(\bfk)$ from the right, where $m=\dim V$. Let $\GL(V)$ act on
 $T_n(V)=V^{\otimes n}$ through diagonal, that is,
 $$
(a_1\cdots a_n)\cdot g=(a_1g)\cdots (a_ng)
$$
for $a_i\in V$ and $g\in \GL(V)$. Recall that the \textit{Schur algebra} is defined by
$$
S(V,n)=\End_{\bfk(\Sigma_n)}(V^{\otimes n}),
$$
where the left action of $\Sigma_n$ on $V^{\otimes n}$ is given by permuting factors. By the classical Schur-Weyl duality, the group $\GL(V,n)$ generates algebra $S(V,n)=\End_{\bfk(\Sigma_n)}(V^{\otimes n})$ and so there is an epimorphism of rings
$$
\bfk(\GL(V))\longrightarrow S(V,n).
$$
Observe that if $M$ is a sub-quotient of a direct sum of copies of $V^{\otimes n}$, then the $\bfk(\GL(V))$-action factors through its quotient algebra $S(V,n)$. Thus if $M$ and $N$ are sub-quotients of direct sums of copies of $V^{\otimes n}$, then
$$
\Hom_{\bfk(\GL(V))}(M,N)=\Hom_{S(V,n)}(M,N).
$$
Recall from~\cite{Green} that the category of $\bfk(\GL(V))$-modules that are sub-quotients of direct sums of copies of $V^{\otimes n}$ is equivalent to the category of modules over the Schur algebra $S(V,n)$, which is denoted by $\Mod(S(V,n))$.

Let $B$ be a sub-quotient of a direct sum of copies of $T_n$. The action of $\GL(V)$ on $V$ induces an action on $B(V)$ via the functor $B$. Thus $B(V)$ is a module over $\bfk(\GL(V))$. Since $B(V)$ is a sub-quotient of a direct sum of copies of $V^{\otimes n}$, $B(V)$ is an object in $\Mod(S(V,n))$. Thus we have a functor:
$$
\begin{array}{rl}
\Theta\colon &B\mapsto B(V)\\
& \Hom(A, B)\longrightarrow \Hom_{\bfk(\GL(V))}(A(V),B(V))=\Hom_{S(V,n)}(A(V),B(V))\\
\end{array}
$$
from the category of sub-quotients of direct sums of copies of $T_n$ to $\Mod(S(V,n))$.

\begin{lem}\label{lemma3.6}
Let $B$ be a sub-quotient of a free $T_n$-functor and let $V$ be a module with $\dim(V)\geq n$. Then $B=0$ if and only if $B(V)=0$.
\end{lem}
\begin{proof}
If $B=0$, clearly $B(V)=0$. Assume that $B(V)=0$. Let $B=\tilde B/B'$ with $B'\hookrightarrow \tilde B \hookrightarrow F$, where $F$ is a direct sum of copies of $T_n$. It is routine to check that $\gamma_j(T_n)=0$ for $j>n$. Thus $\gamma_j(F)=0$ for $j>n$ and so
$$
\gamma_j(B')=\gamma_j(\tilde B)=\gamma_j(B)=0
$$
for $j>n$. Since $B(V)=0$, we have $B(\bar V_n)=0$ because $\dim V\geq \dim\bar V_n=n$. Thus $\gamma_j(B)=0$ for $j\leq n$. The assertion follows by Corollary~\ref{corollary2.6}.
\end{proof}

\begin{cor}\label{corollary3.7}
Let $A$ and $B$ sub-quotients of free $T_n$-functors and let $V$ be a module with $\dim(V)\geq n$. Then
$$
\Theta\colon \Hom(A,B)\longrightarrow \Hom_{\bfk(\GL(V))}(A(V),B(V))=\Hom_{S(V,n)}(A(V),B(V))
$$
is a monomorphism.
\end{cor}
\begin{proof}
Let $f\colon A\to B$ be a natural transformation such that $f_V\colon A(V)\to B(V)$ is $0$. Let $C=\mathrm{Im}(f\colon A\to B)$. Then $C(V)=0$. Thus $C=0$ and hence the result.
\end{proof}

A direct sum of finite copies of $T_n$ is called a \textit{finite free $T_n$-functor}.
\begin{prop}\label{proposition3.8}
Let $B$ be a sub-quotient of a finite free $T_n$-functor and let $A$ be a quotient functor of a finite free $T_n$-functor. Suppose that $\dim V\geq n$. Then
the homomorphism
$$
\Theta_{A,B}\colon \Hom(A,B)\longrightarrow \Hom_{\bfk(\GL(V))}(A(V),B(V))=\Hom_{S(V,n)}(A(V),B(V))
$$
is an isomorphism.
\end{prop}
\begin{proof}
By Schur-Weyl duality, the monomorphism
$$
\Theta_{T_n,T_n}\colon \Hom(T_n,T_n)\longrightarrow \Hom_{\bfk(\GL(V))}(T_n(V),T_n(V))
$$
is an epimorphism and so $\Theta_{A,B}$ is an isomorphism when $A$ and $B$ are free $T_n$-functors. According to~\cite[p.94]{Donkin}, $V^{\otimes n}$ is projective over $S(V,n)$. Let $A$ be a free $T_n$ functor. By tracking the exact sequence for
$$
\Theta_{A,-} \colon \Hom(A,-)\longrightarrow \Hom_{S(V,n)}(A(V), -)
$$
together with the fact that $\Theta_{A,B}$ is always a monomorphism, we have
$$
\Theta_{A,B} \colon \Hom(A,B)\longrightarrow \Hom_{S(V,n)}(A(V), B(V))
$$
for any sub-quotient $B$ of a free $T_n$-functor. Let $A$ be a quotient of a free functor $F$ with an epimorphism $\phi\colon F\to A$. Let $C=\Ker(\phi)$. From the commutative diagram of exact sequences
\begin{diagram}
\Hom(A,B)&\rInto^{\phi^*}&\Hom(F, B)&\rTo&\Hom(C,B)\\
\dInto>{\Theta_{A,B}}&&\cong\dTo>{\Theta_{F,B}}&&\dInto>{\Theta_{C,B}}\\
\Hom_{S(V,n)}(A(V),B(V))&\rInto^{\phi^*}&\Hom_{S(V,n)}(F(V), B(V))&\rTo&\Hom_{S(V,n)}(C(V),B(V)),\\
\end{diagram}
the monomorphism $\Theta_{A,B}$ is an epimorphism and hence the result.
\end{proof}

A $\bfk(\GL(V))$-submodule $M$ of $L_n(V)$ is called \textit{$T_n$-projective} if $M$ is isomorphic to a summand of a direct summation of $T_n(V)$ as modules over $\bfk(\GL(V))$. Let
$\bar\beta_n\colon T_n\to L_n$ be the natural epimorphism defined by
$
\bar\beta_n(a_1\cdots a_n)=[[a_1,a_2],\ldots,a_n]
$
for any module $W$ and any monomial $a_1\cdots a_n\in T_n(W)=W^{\otimes n}$.  Let $\beta_n$ be the composite
$
T_n\rOnto^{\bar\beta_n} L_n\rInto^{i}T_n.
$

\begin{thm}\label{theorem3.9}
Suppose that $\dim V\geq n$. Then
\begin{enumerate}
\item The ring homomorphism
$$
\Theta_{L_n,L_n}\colon \Hom(L_n,L_n)\longrightarrow \Hom_{\bfk(\GL(V))}(L_n(V),L_n(V))
$$
is an isomorphism.
\item There is a one-to-one correspondence, multiplicity preserving between summands of the functor $L_n$ and $\bfk(\GL(V))$-summands of $L_n(V)$.
\item There is a one-to-one correspondence, multiplicity preserving between $T_n$-projective subfunctors of $L_n$ and $T_n$-projective $\bfk(\GL(V))$-submodules of $L_n(V)$.
\item For the functor $L^{\max}_n$, the module $L_n^{\max}(V)$ is the maximum $T_n$-projective submodule of $L_n(V)$ in the sense that any $T_n$-projective $\bfk(\GL(V))$-submodule of $L_n(V)$ is isomorphic to a summand of $L_n^{\max}(V)$.
\item For any choice of the functor $L_n^{\max}$, the socle $\mathrm{Soc}(L_n^{\max}(V))$ is uniquely determined by
$$
\bar\beta_n(\mathrm{Soc}(V^{\otimes n}))=\beta_n(\mathrm{Soc}(V^{\otimes n})).
$$
as the submodule of $L_n(V)\subseteq T_n(V)$.
\item For any choice of the functor $L_n^{\max}$, the head $\mathrm{Hd}(L_n^{\max}(V))$ is uniquely determined by
$$
\beta_n(\mathrm{Hd}(V^{\otimes n})).
$$
as the quotient module of $T_n(V)$.
\end{enumerate}
\end{thm}
\textbf{Note.} By assertion (5), the functor $L_n^{\max}$ and the module $L_n^{\max}(V)$ are determined by evaluating the map $\bar\beta_n$ or $\beta_n$ on simple $\bfk(\GL(V))$-submodules of $V^{\otimes n}$.
\begin{proof}
Since $L_n$ is a quotient functor of $T_n$, assertion (1) is direct consequence of Proposition~\ref{proposition3.8}. Assertion (2) follows from (1) immediately. Assertion (4) is a direct consequence of 3. The proof of assertion (6) is similar to that of assertion (5).

For proving assertion (3), let $M$ be a $T_n$-projective $\bfk(\GL(V))$-submodule of $L_n(V)$. According to~\cite[p.94]{Donkin}, $T_n(V)$ is an injective module over $S(V,n)$ and so is $M$ because $M$ is a summand of a direct sum of finite copies of $T_n$. Thus the inclusion
$$
j\colon M \hookrightarrow L_n(V) \hookrightarrow T_n(V)
$$
admits a $\bfk(\GL(V))$-retraction $r\colon T_n(V)\to M$. The composite $$e=j\circ r\colon T_n(V)\to T_n(V)$$ is an idempotent in
$$
\End_{\bfk(\GL(V))}(T_n(V)).
$$
The natural transformation $\alpha=\Theta^{e}\colon T_n\to T_n$ is an idempotent. Let $B=\mathrm{Im}(\alpha)$. Then $B$ is a $T_n$-projective subfunctor of $L_n$ with $B(V)=M$ and hence assertion (3).

(5). Let
$$
V^{\otimes n}=\bigoplus_{i\in I} P_i
$$
be a decomposition over $S(V,n)$ such that each $P_i$ is indecomposable. The map $\bar\beta_n\colon V^{\otimes n}\to L_n(V)$ induces a map
$$
\bar\beta_n\colon \mathrm{Soc}(V^{\otimes n}=\bigoplus_{i\in I} \mathrm{Soc}(P_i)\longrightarrow \mathrm{Soc}(L_n(V)).
$$
Note that each indecomposable $S(V,n)$-summand of $V^{\otimes n}$ has a unique socle. (See for instance~\cite[(6.4b)]{Green}.
Thus there exists $I'\subseteq I$ such that
$$
P=\bigoplus_{i\in I'}P_i
$$
has the property that
$$
\bar\beta_n| \colon \mathrm{Soc}(P)\longrightarrow \bar\beta_n(\mathrm{Soc}(V^{\otimes n}))
$$
is an isomorphism. It follows that
$$
\bar\beta_n|\colon P\longrightarrow L_n(V)
$$
is a monomorphism because it restricts to the socle is a monomorphism. Since $P$ is an injective $S(V,n)$-module, the map $\bar \beta_n|_P$ has a retraction. Thus $T_n$-projective $S(V,n)$-module $P$ is isomorphic to a $S(V,n)$-summand of $L_n(V)$. From the maximum property of $L^{\max}(n)$, $P$ is isomorphic to a $S(V,n)$-summand of $L^{\max}_n(V)$. In particular,
$$
\bar\beta_n(\mathrm{Soc}(V^{\otimes n}))=\bar\beta_n|(\mathrm{Soc}(P))\subseteq \mathrm{Soc}(L^{\max}_n(V)).
$$
On the other hand, since $L^{\max}_n(V)$ is $S(V,n)$-projective, the inclusion $$j\colon L^{\max}_n(V)\hookrightarrow L_n(V)$$ admits a $S(V,n)$-lifting $\tilde j\colon L^{\max}_n(V)\to V^{\otimes n}$ such that $j=\bar\beta\circ\tilde j$. Thus
$$
\mathrm{Soc}(L^{\max}_n(V))\subseteq \bar\beta_n(\mathrm{Soc}(V^{\otimes n})).
$$
Note that the inclusion $i\colon L_n(V)\hookrightarrow T_n(V)$ induces a monomorphism $i|\colon \mathrm{Soc}(L_n(V))\hookrightarrow \mathrm{Soc}(T_n(V))$. Thus
$$
\bar\beta_n(\mathrm{Soc}(V^{\otimes n}))=\beta_n(\mathrm{Soc}(V^{\otimes n}))
$$
and hence the result.
\end{proof}

We give a remark that if $\dim V<n$, then assertions (3) and (4) are not true by the following example.

\begin{example}\label{example3.10}
Let $\bfk$ be of characteristic $3$ and let $V$ be a $2$-dimensional module with a basis $\{u,v\}$. Then the canonical map
$$
f\colon L_2(V)\otimes V \longrightarrow L_3(V)\ [a_1,a_2]\otimes a_3\mapsto [[a_1,a_2],a_3]
$$
is an isomorphism of modules over $\bfk(\GL_2(\bfk))$. Since $L_2(V)$ is a $\bfk(\GL(V))$-summand of $V^{\otimes 2}$, $L_2(V)\otimes V$ is $T_3$-projective. Thus $L_3(V)$ is $T_3$-projective. On the other hand, it is easy to see that the functor $L^{\max}_3=0$ and so $L^{\max}_3(V)=0$.

In this case, $L_3(V)$ is not an injective $S(V,3)$-module. In fact, the inclusion $L_3(V)\hookrightarrow V^{\otimes 3}$ does not have $S(V,3)$-retraction by expecting the Steenrod module structure on $V^{\otimes 3}$. Also it is easy to check that $L_3(V)$ is not a projective $S(V,3)$-module.\hfill $\Box$
\end{example}

We call $M\subseteq L_n(V)$ \textit{functorial $T_n$-projective} if there exists a $T_n$-projective subfunctor $Q$ of $L_n$ such that $M=Q(V)$. (\textbf{Note.} Here we require that $Q(V)$ is strictly equal to $M$ rather than isomorphic to $M$.)

\begin{prop}\label{proposition3.12}
Assume that the ground field $\bfk$ has infinite elements. Let $V$ be any $\bfk$-module and let $M$ be a $\bfk(\GL(V))$-submodule of $L_n(V)$. Then $M$ is functorial $T_n$-projective if and only if $M$ satisfies the following two conditions:
\begin{enumerate}
\item[1)] There exists a $\bfk(\GL(V))$-linear map $r\colon V^{\otimes n}\to M$ such that $r|_M$ is the identity.
\item[2)] The inclusion $M\hookrightarrow L_n(V)$ admits the following lifting
\begin{diagram}
& &V^{\otimes n}\\
&\ruDashto &\dOnto>{\bar\beta_n}\\
M&\rInto& L_n(V).\\
\end{diagram}
as modules over $\bfk(\GL(V))$.
\end{enumerate}
\end{prop}
\begin{proof}
Suppose that $M$ is functorial $T_n$-projective. Let $Q$ be a subfunctor of $L_n$ with $Q(V)=M$. The inclusion
$$
Q\rInto L_n\rInto T_n
$$
admits a natural retraction because $\gamma_n(Q)$ is injective. By evaluating at $V$, condition 1 is satisfied. Since $\gamma_n(Q)$ is projective, there a natural lifting $\tilde j\colon Q\to T_n$ such that $\beta_n\circ \tilde j$ is the inclusion of $Q$ in $L_n$ and so condition 2 is satisfied by evaluating at $V$.

Conversely suppose that $M$ satisfies conditions 1 and 2. Let $j\colon M \hookrightarrow L_n(V)$ and $ L_n(V)\hookrightarrow V^{\otimes n}$ be the inclusions. Let $\tilde j\colon M\to V^{\otimes n}$ be a $\bfk(\GL(V))$-map such that $\beta_n\circ \tilde j=j$. By the Schur-Weyl duality, the map
$$
\bfk(\Sigma_n)=\Hom(T_n,T_n)\longrightarrow \End_{\bfk(\GL(V))}(V^{\otimes n})
$$
is an epimorphism. There exists a natural transformation $\alpha\colon T_n\to T_n$ such that $\alpha_V=\tilde j\circ r$.
Let $\theta=\beta_n\circ \alpha$. Consider the colimit of the sequence
$$
T_n\rTo^{\theta} T_n\rTo^{\theta} T_n\rTo\cdots,
$$
there exists $k>>0$ such that
$$
Q=\mathrm{Im}(\theta^k)\longrightarrow \colim_{\theta} T_n
$$
is an isomorphism because, by taking $\gamma_n(-)$ to the above sequence, the submodules $\mathrm{Im}(\gamma_n(\theta^t))$ of $\gamma_n(T_n)$ has the monotone decreasing in dimensions:
$$
\dim \mathrm{Im}(\gamma_n(\theta))\geq \dim \mathrm{Im}(\gamma_n(\theta^2))\geq \dim \mathrm{Im}(\gamma_n(\theta^3))\geq\cdots.
$$
Since $Q=\beta_n(\alpha\theta^{k-1}(T_n))$, $Q$ is a $T_n$-projective subfunctor of $L_n$. By evaluating at $V$, we check that $Q(V)=M$. Since
$$
\begin{array}{rcl}
\theta_V\circ\theta_V&=&\beta_n\circ\alpha_V\circ \beta_n\circ\alpha_V\\
&=&i\circ\bar\beta_n \circ \tilde j\circ r\circ i\circ\bar\beta_n \circ \tilde j\circ r\\
&=&i\circ j\circ r\circ i\circ j\circ r\\
&=& i\circ j\circ r\\
&=& \theta_V,\\
\end{array}
$$
the map $\theta_V$ is an idempotent. It follows that
$$
Q(V)=\mathrm{Im}(\theta_V)=i\circ j\circ r(V^{\otimes n})=M
$$
and hence the result.
\end{proof}

\section{Coalgebra Structure on Tensor Algebras}\label{section4}
In this section, the tensor algebra $T(V)$ admits the comultiplication
$$
\psi\colon T(V)\longrightarrow T(V)\otimes T(V)
$$
described in subsection~\ref{subsection2.1}.
\subsection{Changing Ground-rings}
Some results in the representation theory help us to change the
ground ring. Let $\Z_{(p)}$ be the $p$-local integers. By modular
representation theory of symmetric groups (see, for example
~\cite[Exercise 6.16, p.142]{CR}), any idempotent in
$(\Z/p)(\Sigma_n)$ lifts to an idempotent in $\Z_{(p)}(\Sigma_n)$.
It is well-known ~\cite{James} that any irreducible modules $M$ over
$\Z/p(\Sigma_n)$ is absolutely irreducible, that is, for any
extension field $\bfk$, $M\otimes\bfk$ is irreducible over
$\bfk(\Sigma_n)$. Thus there is a one-to-one correspondence between
idempotents in $\Z/p(\Sigma_n)$ and the idempotents in
$\bfk(\Sigma_n)$.

Let $R$ be any commutative ring with identity. Consider $T\colon
V\mapsto T(V)$ as the functor from projective $R$-modules to
coalgebras over $R$. Denote by $\coalg^R(T,T)$ the set of self
natural coalgebra transformation of $T$. Let $\bfk$ be any field of
characteristic $p$. We have the canonical functions
$$
R\colon \coalg^{\Z_{(p)}}(T,T)\rTo\coalg^{\Z/p}(T,T)
$$
by modulo $p$ and
$$
K\colon \coalg^{\Z/{p}}(T,T)\rTo \coalg^{\bfk}(T,T)
$$
by tensoring with $\bfk$ over $\Z/p$. By~\cite[Corollary 6.9]{SW1},
there is a one-to-one correspondence between natural indecomposable
retract of $T$ over $\bfk$ and the indecomposable
$\bfk(\Sigma_n)$-projective submodule of $\Lie(n)$ for $n\geq
1$. Thus we have the following:
\begin{prop}\label{proposition4.1}
The functions $R$ and $K$ have the following properties:
\begin{enumerate}
\item The map $
R\colon \coalg^{\Z_{(p)}}(T,T)\rTo\coalg^{\Z/p}(T,T)$ induces a
one-to-one correspondence betweens idempotents. Thus every natural
coalgebra decompositions of $T$ over $\Z/p$ lifts to a natural
coalgebra decomposition over $\Z_{(p)}$.
\item  The map $
K\colon \coalg^{\Z_{p}}(T,T)\rTo\coalg^{\bfk}(T,T)$ induces a
one-to-one correspondence betweens idempotents. Thus natural
coalgebra decompositions of $T$ only depends on the characteristic
of the ground field.\hfill $\Box$
\end{enumerate}
\end{prop}

By this proposition, we can freely change the ground fields $\bfk$
with the same characteristic and lift natural coalgebra
decompositions over $p$-local integers if it is necessary.

\subsection{Block Decompositions}

From now on in this section, the ground field $\bfk$ is algebraically closed with
$\mathrm{char}(\bfk)=p$.  For any coalgebra $C$, let $PC$ be the set of the primitives of $C$. If $C$ is a functor from modules to coalgebras, then $PC$ is a functor from modules to modules. Recall from Proposition~\ref{proposition2.3} that any subfunctor of $T$ is a graded subfunctor. Thus if $C$ is sub-quotient coalgebra functor of $T^C$, then $C$ is graded and so we have the homogenous functors $C_n$ and $P_nC=PC\cap C_n$ for each $n$. For the case $C=T$, $PT(V)=L^{\res}(V)$ is the free restricted Lie algebra generated by $V$ and $P_nT=L^{\res}_n$ for each $n$.

For natural transformations $f,g\colon T\to T$, the \textit{convolution product $f\ast g$} is defined by the composite
$$
T(V)\rTo^{\psi} T(V)\otimes T(V)\rTo^{f\otimes g} T(V)\otimes T(V)\rTo^{\mathrm{multi.}} T(V).
$$
If $f$ and $g$ are natural coalgebra transformation, clearly $f\ast g$ is also a natural coalgebra transformation. For any element $\zeta\in\bfk$, define $\lambda_\zeta\colon T(V)\to T(V)$ by setting
\begin{equation}\label{equation4.1}
\lambda_{\zeta}(a_1\cdots a_n)=\zeta^n a_1\cdots a_n
\end{equation}
for $a_1,\ldots,a_n\in V$. In other words, $\lambda_{\zeta}\colon T(V)\to T(V)$ is the (unique) Hopf map such that  $\lambda_{\zeta}(a)=\zeta a$ for $a\in V$. Let $\chi\colon T(V)\to T(V)$ be the conjugation of the Hopf algebra $T(V)$, namely $\chi$ is the anti-homomorphism such that $\chi(a)=-a$ for $a\in V$. More precisely,
\begin{equation}\label{equation4.2}
\chi(a_1\cdots a_n)=(-1)^{n}a_na_{n-1}\cdots a_1
\end{equation}
for $a_1,\ldots,a_n\in V$. For any element $\zeta\in\bfk$, we have the natural coalgebra transformation
\begin{equation}\label{equation4.3}
\theta_{\zeta}=\lambda_\zeta\ast\chi\colon T(V)\to T(V).
\end{equation}
If $\alpha\in P_nT(V)$, then
\begin{equation}\label{equation4.4}
\theta_{\zeta}(\alpha)=(\zeta^n-1)\alpha
\end{equation}
by the definition of convolution product because $\psi(\alpha)=\alpha\otimes 1+1\otimes\alpha$. For general monomials in $T_n(V)$, it is straightforward to have the formula
\begin{equation}\label{equation4.5}
\theta_{\zeta}(a_1\cdots a_n)=\sum
_{
\begin{array}{c}
\sigma(1)<\cdots<\sigma(k)\\
\sigma(k+1)<\cdots<\sigma(n)\\
\sigma\in\Sigma_n\\
0\leq k\leq n\\
\end{array}}
\left(\zeta^k+(-1)^{n-k}\right) a_{\sigma(1)}\cdots a_{\sigma(k)}a_{\sigma(k+1)}\cdots a_{\sigma(n)}.
\end{equation}
The maps $\theta_\zeta$ are useful for obtaining natural coalgebra decompositions of $T(V)$.

\begin{thm}\label{theorem4.2}
Let the ground ring $\bfk$ be a field of characteristic $p$. Then there exists a natural coalgebra decomposition
$$
T(V)\cong C(V)\otimes D(V)
$$
for any module $V$ with the property that
$$
PC_n=\left\{
\begin{array}{lcl}
0&\textrm{ if }& n \textrm{ is not a power of } p,\\
P_nT &\textrm{ if } & n=p^r \textrm{ for some } r.\\
\end{array}\right.
$$
\end{thm}
\textbf{Note.} From the decomposition, we have $P_nD=0$ if $n$ is a power of $p$ and $P_nD=P_nT$ if $n$ is not a power of $p$. The theorem gives a decomposition that one can put all primitives of tensor length of powers of $p$ in a one coalgebra factor and put the rest primitives in another coalgebra factor.
\begin{proof}
Let $\{m_1<m_2<m_3<\cdots\}$ be the set of all positive integers prime
to $p$ excluding $1$ and let $\zeta_{m_i}$ be a primitive $m_i$th
root of $1$. We are going to construct by induction a sequence of sub coalgebra functor $C(k)$ of $T$, with the inclusion denoted by $j_k\colon C(k) \hookrightarrow T$, and a sequence of quotient coalgebra functor $q_k\colon T\to E(k)$ with the following properties:
\begin{enumerate}
\item[1)] $C(k+1)$ is a subfunctor $C(k)$ for each $k\geq 0$.
\item[2)] There exists a coalgebra natural transformation $q'_k\colon E(k)\to E(k+1)$ such that $q_{k+1}=q'_k\circ q_k$ for each $k\geq 0$.
\item[3)] The composite $q_k\circ j_k\colon C(k)\to E(k)$ is a natural isomorphism.
\item[4)] The primitives $P_nC(k)=0$ if $n$ is divisible by one of $m_1,m_2,\ldots,m_k$.
\item[5)] The primitive $P_nC(k)=P_nT$ if $n$ is not divisible by any of $m_1,m_2,\ldots,m_k$.
\end{enumerate}
Let $C(0)=E(0)=T$ and let $i_0=q_0=\id$. The construction of $C(1)$ and $E(1)$ is as follows. Let $E(1)=\colim_{\theta_{\zeta_{m_1}}}T$ be the colimit of the sequence of coalgebra natural transformation
$$
T\rTo^{\theta_{\zeta_{m_1}}} T\rTo^{\theta_{\zeta_{m_1}}} T\rTo\cdots.
$$
Let $q_1\colon T\to E(1)$ be the map to its colimit. By~\cite[Theorem 4.5]{SW1}, there exists a sub coalgebra functor $C(1)$ of $T$, with the inclusion denoted by $j_1\colon C(1)\to T$, such that $q_1\circ i_1$ is a natural isomorphism. From Equation~\ref{equation4.4},
$$
\theta_{\zeta_{m_1}}\colon P_nT\to P_nT
$$
is zero if $m_1|n$ and an isomorphism if $m_1\nmid n$. Thus $P_nE(1)=\colim_{\theta_{\zeta_{m_1}}}P_nT=0$ if $m_1|n$ and
$$
q_1\colon P_nT\longrightarrow P_nE(1)
$$
is an isomorphism if $m_1 \nmid n$. Since $C(1)\cong E(1)$, conditions (4) and (5) holds. Now suppose that we have construct $C(j)$ and $E(j)$ satisfying conditions (1)-(5) for $j\leq k$. Let $f\colon T\to T$ be the composite
$$
T\rOnto^{q_k}E_n\rTo^{(q_k\circ j_k)^{-1}}_{\cong} C(k)\rInto^{j_k}T \rTo^{\theta_{\zeta_{m_{k+1}}}}T
$$
and let $E(k+1)=\colim_{f}T$. Let $q_{k+1}\colon T\to E(k+1)$ be the canonical map to its colimit. Notice that
$$
q_{k+1}\circ f=q_{k+1}\colon T\longrightarrow E(k+1).
$$
Let $q'_k=q_{k+1}\circ j_k\circ (q_k\circ j_k)^{-1}$. Then $q_{k+1}=q'_k\circ q_k$ and so condition (2) satisfies. Since $f$ factors through the subfunctor $C(k)$, there exists a subfunctor $C(k+1)$ of $C(k)$, with the inclusion into $T$ denoted by $j_{k+1}$, such that $q_{k+1}\circ j_{k+1}$ is a natural isomorphism. Hence we have conditions (1) and (3). Let $\alpha\in P_nT(V)$. Then
$$
f(\alpha)=\theta_{\zeta_{m_{k+1}}}((j_k\circ(q_k\circ j_k)^{-1}\circ q_k)(\alpha))=(\zeta_{m_{k+1}}^n-1)((j_k\circ(q_k\circ j_k)^{-1}\circ q_k)(\alpha).
$$
Thus $f(\alpha)=0$ if $n$ is divisible by one of $m_1,\ldots,m_{k+1}$ and
$$
f\colon P_nT\longrightarrow P_nT
$$
is an isomorphism if $n$ is not divisible by any of $m_1,\ldots,m_{k+1}$. It follows that $P_nE(k+1)=0$ if $n$ is divisible by one of $m_1,\ldots,m_{k+1}$ and
$$
q_{k+1}\colon P_nT\longrightarrow P_nE(k+1)
$$
is an isomorphism if $n$ is not divisible by any of $m_1,\ldots,m_{k+1}$. Since $C(k+1)\cong E(k+1)$, we have conditions (4) and (5). The induction is finished.

Now let
$$
C=\bigcap_{k=0}^\infty C(k)
$$
be the intersection of the subfunctors $C(k)$ of $T$ and let $E(\infty)$ be the colimit of the sequence
$$
T\rOnto^{q_1}E(1)\rOnto^{q'_2}E(2)\rOnto^{q'_3}E(3)\rOnto\cdots.
$$
From condition (3), each $C(k)$ is coalgebra retract of $T$ and so each $C(k)$ is a functor from modules to coassociative and cocommutative quasi-Hopf algebras with the multiplication on $C(k)$ given by
$$
C(k)\otimes C(k)\rInto T\otimes T\rTo^{\mathrm{multi}} T\rOnto C(k),
$$
where we use the notation of quasi-Hopf algebra given in~\cite{MM}. By conditions (1)-(3), $C(k+1)$ is a coalgebra retract of $C(k)$ and so there is a natural coalgebra decomposition
$$
C(k)\cong C(k+1)\otimes C'(k)
$$
by~\cite[Lemma 5.3]{SW1}. From conditions (4) and (5), $P_nC(k+1)=P_nC(k)$ for $n<m_{k+1}$ and so $P_nC'(k)=0$ for $n<m_{k+1}$. It follows that
$$
C'(k)_n=0
$$
for $0<n<m_{k+1}$. Thus
$$
C(k+1)_n=C(k)_n
$$
for $n<m_{k+1}$ and from conditions (1)-(3),
$$
q_k\colon E(k)_n\longrightarrow E(k+1)_n
$$
is an isomorphism for $n<m_{k+1}$. Notice that the integers $m_k\to\infty$ as $k\to\infty$. Let $n$ be a fixed positive integer. For the integers $k$ with $m_k>n$, we have $C_n=C(k)_n$ and
$$
E(k)_n\rTo^{q'_k}_{\cong} E(k+1)_n\rTo^{q'_{k+1}}_{\cong} E(k+2)_n\rTo.
$$
It follows that the composite
$$
C_n=C(k)_n \rInto^{j_k} T_n\rOnto^{q_k} E(k)_n\rTo E(\infty)_n
$$
is an isomorphism. Thus the composite
$$
C\rInto T\rTo E(\infty)
$$
is an isomorphism and so $C$ is a coalgebra retract of $T$. This gives a natural coalgebra decomposition
$$
T\cong C\otimes D
$$
for some coalgebra retract $D$ of $T$. From conditions (4) and (5), we have $P_nC=0$ if $n$ is not a power of $p$ and $PC_{p^r}=PT_{p^r}$ for $r\geq 0$. The proof is finished.
\end{proof}

\begin{thm}[Block Decomposition Theorem]\label{theorem4.3}
Let $\bfk$ be a field of characteristic $p$.
Let $\{m_i\}_{i\geq0}$ be the set of all positive integers prime to $p$
with the order that $m_0=1<m_1<m_2<\cdots$. Then there exist natural coalgebra retracts $C^{m_i}$ of $T$ with a natural coalgebra decomposition
$$
T(V)\cong\bigotimes_{i=0}^\infty C^{m_i}(V)
$$
such that
$$
P_nC^{m_i}=\left\{
\begin{array}{ccc}
P_nT &\textrm{ if } & n=m_ip^r \textrm{ for some } r\geq0,\\
0&&\textrm{ otherwise }.\\
\end{array}\right.
$$
\end{thm}
\begin{proof}
We are going to show by induction that there exists natural coalgebra retracts $C^{m_i}$ of $T$, for $0\leq i\leq k$, with a natural coalgebra decomposition
\begin{equation}\label{equation4.6}
T(V)\cong \left( \bigotimes_{i=0}^k C^{m_i}(V)\right)\otimes D^k(V)
\end{equation}
for some natural coalgebra retract $D^k$ of $T$ such that
$$
P_nC^{m_i}=\left\{
\begin{array}{ccc}
P_nT &\textrm{ if } & n=m_ip^r \textrm{ for some } r\geq0,\\
0&&\textrm{ otherwise }.\\
\end{array}\right.
$$
for $0\leq i\leq k$. The statement holds for $k=0$ by Theorem~\ref{theorem4.2}, where $C^{m_0}$ is the natural coalgebra retract $C$ of $T$ given in Theorem~\ref{theorem4.2}. Suppose that the statement holds for $k$. Following the lines in the proof of Theorem~\ref{theorem4.2} with using $\{\theta_{\zeta_{m_i}}\}$ for $i\geq k+2$, there is a natural coalgebra decomposition
$$
D^k(V)\cong C^{m_{k+1}}(V)\otimes D^{k+1}(V).
$$
In brief, the first construction of $E^{m_{k+1}}(1)=\colim_gT$ is the colimit of the map $g$ given by the composite
$$
T\rOnto D^k \rInto T\rTo^{\theta_{\zeta_{m_{k+2}}}}T
$$
and the inductive construction follows the lines in the proof of Theorem~\ref{theorem4.2} with pre-composing $T\twoheadrightarrow D^k\hookrightarrow T$. This gives a monotone decreasing sequence of natural coalgebra retracts $C^{m_{k+1}}(i)$ of $D^k$ for $i=1,2,\ldots$ and  the resulting natural coalgebra retract $C^{m_{k+1}}=\bigcap\limits_{i=1}^\infty C^{m_{k+1}}(i)$ of $D^k$ has the property in primitives that $P_nC^{m_{k+1}}=0$ if $n$ is divisible by one of $m_{k+2},m_{k+3},\ldots$ and $$P_nC^{m_{k+1}}=P_nD^k$$ if $n$ is not divisible by any of $m_{i}$ with $i\geq k+2$. Together with the fact that $P_nD=0$ if $n=m_ip^r$ for some $0\leq i\leq k$ and $r\geq0$ and $P_nD=P_nT$ otherwise, we have $P_nC^{m_{k+1}}=P_nT$ if $n=m_{k+1}p^r$ for some $r\geq 0$ and $P_nC^{m_{k+1}}=0$ otherwise. The induction is finished.

Now decomposition~\ref{equation4.6} induces a commutative diagram
\begin{diagram}
T\cong \left(\bigotimes_{i=0}^k C^{m_i}\right)\otimes D^k\cong \left(\bigotimes_{i=0}^{k+1} C^{m_i}\right)\otimes D^{k+1} &\rOnto^{q_{k+1}}_{\mathrm{proj.}}& \bigotimes_{i=0}^{k+1} C^{m_i}\\
&\rdOnto^{q_k}&\dOnto>{\mathrm{proj.}}\\
&& \bigotimes_{i=0}^k C^{m_i}\\
\end{diagram}
that induces a coalgebra natural transformation
$$
T\rTo^q \bigotimes_{i=0}^\infty C^{m_i}
$$
which is an isomorphism because, for each $n$,  $D^k_n=0$ for sufficiently large $k>>0$. We finish the proof.
\end{proof}

\section{Proof of Theorem~\ref{theorem1.1}}\label{section5}

In this section, the ground ring is a field $\bfk$ of characteristic $p$.

\begin{lem}\label{lemma5.1}
Let $Q$ be a $T_n$-projective sub functor of $L^{\res}_n$. Then $Q$ is a subfunctor of $L_n$ and the sub
Hopf algebra $T(Q(V))$ of $T(V)$ generated by $Q(V)$ is a natural
coalgebra retract of $T(V)$.
\end{lem}
\begin{proof}
It is easy to see that $\gamma_n(L^{\res}_n)=\gamma_n(L_n)=\Lie(n)$. By Proposition~\ref{proposition3.1}, the image of the natural transformation
$$
\Phi^{L^{\res}_n}_V\colon \gamma_n(L^{\res}_n)\otimes_{\bfk(\Sigma_n)}V^{\otimes n}\longrightarrow L^{\res}_n(V)
$$
is $L_n(V)$. Since $Q$ is $T_n$-projective,
$$
\Phi^Q_V\colon \gamma_n(Q)\otimes_{\bfk(\Sigma_n)}V^{\otimes n}\longrightarrow Q(V)
$$
is an isomorphism by Proposition~\ref{proposition2.10} (1). From the commutative diagram
\begin{diagram}
\gamma_n(Q)\otimes_{\bfk(\Sigma_n)}V^{\otimes n}&\rTo& \gamma_n(L^{\res}_n)\otimes_{\bfk(\Sigma_n)}V^{\otimes n}\\
\cong\dTo>{\Phi^Q_V}&&\dTo>{\Phi^{L^{\res}_n}_V}\\
Q(V)&\rInto& L^{\res}_n(V),\\
\end{diagram}
we have $Q\subseteq L_n$. By Proposition~\ref{proposition2.10}(4), there is a natural linear transformation
$$
r_V\colon T_n(V)=V^{\otimes
n}\rTo Q(V)
$$
with $r_V|_{Q(V)}=\id_{Q(V)}$. Let
$$
H_n\colon T(V)\rTo T(V^{\otimes n})
$$
be the algebraic James-Hopf map induced by the geometric James-Hopf
map by taking homology. Then there is a commutative diagram
\begin{diagram}
T(Q(V))&\rInto& T(L_n(V))&\rInto& T(V)\\
&\rdEq&&\rdTo>{T(j_V)}&\dTo>{H_n}\\
&&T(Q(V))&\lTo^{T(r_V)} &T(V^{\otimes n}),\\
\end{diagram}
where the maps in the top row are the inclusions of sub Hopf
algebras, $j_V$ is the canonical inclusion and the right triangle
commutes by the geometric theorem in~\cite[Theorem 1.1]{Wu1}. Thus
the sub Hopf algebra $T(Q(V))$ of $T(V)$ admits a natural coalgebra
retraction and hence the result.
\end{proof}

A natural sub Hopf algebra $B(V)$ of $T(V)$ is called
\textit{coalgebra-split} if the inclusion $B(V)\to T(V)$ admits a
natural coalgebra retraction. For a Hopf algebra $A$, denote by $QA$
the set of indecomposable elements of $A$. Let $IA$ be the
augmentation ideal of $A$. If $B(V)$ is a natural sub Hopf algebra
of $T(V)$, then there is a natural epimorphism $IB(V)\to QB(V)$. Let
$Q_nB(V)$ be the quotient of $B_n(V)=IB(V)\cap T_n(V)$ in $QB(V)$.

\begin{thm}\label{theorem5.2}
Let $B(V)$ be a natural sub Hopf algebra of $T(V)$. Then the
following statements are equivalent to each other:
\begin{enumerate}
\item[1)] $B(V)$ is a natural coalgebra-split sub Hopf algebra of $T(V)$.
\item[2)] There is a natural linear transformation $r\colon T(V)\to
B(V)$ such that $r|_{B(V)}$ is the identity.
\item[3)] Each $Q_nB$ is naturally equivalent to a $T_n$-projective sub functor of $L_n$.
\item[4)] Each $Q_nB$ is a $T_n$-projective functor.
\end{enumerate}
\end{thm}

\begin{proof}
$(1)\Longrightarrow (2)$ and $(3)\Longrightarrow (4)$ are obvious. By~\cite[Theorem 8.6]{SW1},
$(2)\Longrightarrow (1)$. Thus $(1)\Longleftrightarrow (2)$. From
the proof of~\cite[Theorem 8.8]{SW1}, $(2)\Longrightarrow (3)$.

$(4)\Longrightarrow (2)$. Since $B(V)$ is a sub Hopf algebra of
primitively generated Hopf algebra $T(V)$, $B(V)$ is primitively
generated and so
$$
r_n\colon P_nB(V)=B(V)\cap L^{\res}_n(V)\rTo Q_nB(V)
$$
is a natural epimorphism, where $L^{\res}(V)=PT(V)$ is the free
restricted Lie algebra generated by $V$. Since $Q_nB$ is $T_n$-projective, the map
$r_n$ admits a natural cross-section $s_n\colon Q_nB(V)\rInto
P_nB(V)$ by Proposition~\ref{proposition2.10} (4).

Now we show that the inclusion $B(V)\to T(V)$ admits a natural
linear retraction. By identifying $Q_nB(V)$ with $s_n(Q_nB(V))$, we
have
$$
B(V)=T\left(\bigoplus_{k=1}^\infty Q_kB(V)\right)\subseteq T(V).
$$
Since each $Q_kB$ is a retract of the functor $T_k$,
$$
Q_{i_1}B\otimes\cdots\otimes Q_{i_t}B
$$
is a retract of $T_{i_1+i_2+\cdots+i_t}$ for any sequence
$(i_1,\ldots,i_t)$.
Note that $\{Q_iB(V) \ | \ i\geq 1 \}$ are algebraically
independent. Thus the summation
$$
\sum_{i_1+i_2+\cdots+i_t=q} Q_{i_1}B(V)\otimes
Q_{i_2}B(V)\otimes\cdots \otimes Q_{i_t}B(V)\subseteq T_q(V)= V^{\otimes q}.
$$
is a direct sum. From the fact that
$$\bigoplus_{i_1+i_2+\cdots+i_t=q} Q_{i_1}B\otimes
Q_{i_2}B\otimes\cdots \otimes Q_{i_t}B$$
is $T_n$-projective, there is
natural linear retraction
$$
V^{\otimes q}\rTo \bigoplus_{i_1+i_2+\cdots+i_t=q}
Q_{i_1}B(V)\otimes Q_{i_2}B(V)\otimes\cdots \otimes Q_{i_t}B(V)
$$
for any $q\geq 1$. Hence the inclusion $B(V)\to T(V)$ admits a
natural linear retraction.
\end{proof}

\begin{proof}[Proof of Theorem~\ref{theorem1.1}]
Let $B(V)$ be the sub Hopf algebra of $T(V)$
generated by
$$
L_{m_ip^r}(V)\quad \textrm{ for } i\in I,\ r\geq 0.
$$
Let $\{n_j\}_{j\geq 1}=\{m_ip^r\ | \ i\geq 1,\ r\geq0\}$ with $$n_1=m_1<n_2<\cdots.$$ Namely we rewrite the integers $m_ip^r$ in order. Let $B[k](V)$ be the sub Hopf algebra of $V$ generated by $L_{n_j}(V)$ for $1\leq j\leq k$. By Theorem~\ref{theorem5.2}, it suffices to show that $Q_nB$ is $T_n$-projective for $n\geq 1$. Let $n$ be a fixed positive integer. Choose $k$ such that $n_k\geq n$. Then the inclusion $B[k]\hookrightarrow B$ induces an isomorphism
$$
Q_nB[k]\cong Q_nB
$$
because $B[k]$ and $B$ has the same set of generators in tensor length $\leq n$. Thus it suffices to show the following statement:
\begin{enumerate}
\item[] \textit{For each $k\geq 1$, $B[k]$ is coalgebra-split.}
\end{enumerate}

The proof of this statement is given by induction on $k$. The statement holds for $k=1$ by Lemma~\ref{lemma5.1} because $L_{m_1}$ is $T_{m_1}$-projective by~\cite[Corollary 6.7]{SW1} from the assumption that $m_i$ is prime to~$p$. Suppose that $B[k-1]$ is coalgebra-split. Thus there is a coalgebra natural transformation $r\colon T\to B[k-1]$ such that $r|_{B[k-1]}$ is the identity map. Let $n_k=m_ip^r$ for some $i$ and $r$. Let $C^{m_i}$ be the natural coalgebra retract in Theorem~\ref{theorem4.3} with a natural coalgebra retraction $r_C\colon T\to C^{m_i}$. Define $f\colon T\to T$ to be the composite
\begin{equation}\label{equation5.1}
T\rTo^{r_C} C^{m_i}\rInto T\rTo^r B[k-1]\rInto T.
\end{equation}
Let $\tilde E=\colim_f T$ be the colimit with the canonical map
$$
q\colon T\longrightarrow \tilde E.
$$
As in lines of the proof of Theorem~\ref{theorem4.2}, there exists a coalgebra subfunctor $\tilde C$ of $C^{m_i}$ such that
$$
q|_{\tilde C}\colon \tilde C\longrightarrow \tilde E
$$
is an isomorphism by~\cite[Theorem 4.5]{SW1} with a natural coalgebra decomposition
\begin{equation}\label{equation5.2}
C^{m_i}\cong \tilde C\otimes \tilde D.
\end{equation}
According to~\cite[Lemma 5.3]{SW1}, the subfunctor $\tilde D$ of $C^{m_i}$ can be chosen as the cotensor product
$\bfk\Box_{\tilde E}C^{m_i}$ under the coalgebra map
$$
q|_{C^{m_i}}\colon C^{m_i}\longrightarrow \tilde E
$$
and so there is a left exact sequence
\begin{equation}\label{equation5.3}
P_n\tilde D\rInto P_nC^{m_i} \rTo^{P_nq|_{C^{m_i}}} \tilde E.
\end{equation}
 for any $n$.

By restricting the map $f$ as the composite in equation~(\ref{equation5.1}) to the primitives, we have the map
$$
P_nf\colon  P_nT\rTo^{P_nr_C} P_nC^{m_i}\rInto P_nT\rTo^{P_nr} P_nB[k-1]\rInto P_nT.
$$
If $n\not=m_ip^t$ for $t\geq 0$, then $P_nf=0$ because $P_nC^{m_i}=0$. If $n=m_ip^t$ for some $t\geq 0$ with $n<n_k$, then $P_nf$ is the identity map because $P_nC^{m_i}=P_nT$ and $P_nB[k-1]=P_nT=L^{\res}_n$ as the sub Hopf algebra $B[k-1]$ contains $L_{m_ip^s}$ for $s\geq0$. Thus
$$
P_n\tilde C=P_nC^{m_i}
$$
for $n<n_k$. From decomposition~(\ref{equation5.2}), we have
\begin{equation}\label{equation5.4}
P_nC^{m_i}=P_n\tilde C\oplus P_n\tilde D
\end{equation}
for all $n$ and so $P_n\tilde D=0$ for $n<n_k$. It follows that $\tilde D_n=0$ for $0<n<n_k$ and
\begin{equation}\label{equation5.5}
\tilde D_{n_k}=P_{n_k}\tilde D.
\end{equation}
Now consider the case $P_nf$ for $n=n_k=m_ip^r$. Since $P_nC^{m_i}=P_nT$, $P_nr_C=\id$ and so $P_nf\circ P_nf=P_nf$ with
$$
P_nf(\alpha)=P_nr(\alpha)
$$
for $\alpha\in P_nT$. Thus the composite
$$
\mathrm{Im}(P_nf)= P_nB[k-1]\rTo^{P_nq|_{B[k-1]}} P_n\tilde E=\colim_{P_nf}P_nT
$$
is an isomorphism. From exact sequence~(\ref{equation5.3}), we have
$$
P_n\tilde D=\Ker(P_nr \colon P_nC^{m_i}=P_nT\to P_nB[k-1]).
$$
Let $j\colon P_nB[k-1]\hookrightarrow P_nC^{m_i}=P_nT$ be the inclusion. From the commutative diagram
\begin{diagram}
& & P_n\tilde D=\Ker(P_nr)& &\\
& &\dInto&\rdTo^{\cong}&\\
P_n\tilde C&\rInto^i& P_nC^{m_i}&\rOnto&\Coker(i)\\
&\rdTo^{\cong}&\dOnto>{P_nr}&&\\
&&P_nB[k-1],&&\\
\end{diagram}
the summation
$P_nB[k-1]+ P_n\tilde D $
in $P_nT=P_nC^{m_i}$ is a direct sum with the decomposition
$$
P_nC^{m_i}=P_nT=P_nB[k-1]\oplus P_n\tilde D.
$$
From definition of $B[k]$, $PB[k](V)$ is the restricted sub Lie algebra of $L^{\res}(V)=PT(V)$ generated by $L_{n_i}(V)$ for $1\leq i\leq k$. It follows that $P_nB[k]=L_n^{\res}=P_nT=P_nC^{m_i}$ and
$$
Q_nB[k]\cong P_nB[k]/P_nB[k-1]=P_nC^{m_i}/P_nB[k-1]\cong P_n\tilde D.
$$
From decomposition~(\ref{equation5.2}), $\tilde D_n$ is a natural summand of $C^{min}_n$. Since $C^{min}$ is a coalgebra retract of $T$, $C^{\min}_n$ is a natural summand of $T_n$. Thus $\tilde D_n$ is $T_n$-projective. By identity~(\ref{equation5.5}), $P_n\tilde D$ is $T_n$-projective. Thus $Q_nB[k]$ is $T_n$-projective. By Theorem~\ref{theorem5.2}, $B[k]$ is coalgebra-split. The induction is finished and hence the result.
\end{proof}

By inspecting the proof, we have the following slightly stronger statement.

\begin{thm}\label{theorem5.3}
Let $\mathcal{M}=\{m_i\}_{i\in I}$ be finite or infinite set of positive integers prime to $p$ with each $m_i>1$. Let $f\colon I\to \{0,1,2,\ldots\}\cup\{\infty\}$ be a function. Then the sub Hopf algebra $B^{\mathcal{M},f}(V)$ of $T(V)$
generated by
$$
L_{m_ip^r}(V)\quad \textrm{ for } i\in I,\  0\leq r< f(i)
$$
is natural coalgebra-split. \hfill $\Box$
\end{thm}

\section{Decompositions of Lie Powers}\label{section6}

Let $m=kp^r$ with $k\not\equiv 0\mod{p}$ and $k>1$. According to~\cite[Theorem 10.7]{SW1}, the functor $L_{kp^r}$ admits the following functorial decomposition:
$$
L_{kp^r}=L'_{kp^r}\oplus L_p(L'_{kp^{r-1}})\oplus\cdots\oplus L_{p^r}(L'_k)
$$
for each $r\geq 0$ starting with $L'_k=L_k$, where each $L'_{kp^r}$ is a summand of $T_{kp^r}$ which is called $T_{kp^r}$-projective in our terminology. By evaluating on $V$, one gets a decomposition of the $\bfk(\GL(V))$-module $L_{kp^r}(V)$ given ~\cite[Theorem 4.4]{Bryant-Schocker} using a different approach from representation theory, where $L'_{kp^r}(V)$ was denoted by $B_{kp^r}$ in~\cite{Bryant-Schocker}. From Theorem~\ref{theorem5.3}, we can obtain various new decompositions of $L_{kp^r}$ and so the decompositions of the $\bfk(\GL(V))$-module $L_{kp^r}(V)$ by evaluating on $V$.

Let $\mathcal{M}=\{m_i\}_{i\in I}$ be finite or infinite set of positive integers prime to $p$ with each $m_i>1$. Let $f\colon I\to \{0,1,2,\ldots\}\cup\{\infty\}$ be a function. Let $B^{\mathcal{M},f}(V)$ be the sub Hopf algebra of $T(V)$
generated by
$$
L_{m_ip^r}(V)\quad \textrm{ for } i\in I,\  0\leq r< f(i).
$$
According to Theorem~\ref{theorem5.3}, $B^{\mathcal{M},f}$ is coalgebra-split and so
$Q_nB^{\mathcal{M},f}$ is $T_n$-projective by Theorem~\ref{theorem5.2}. Since $B^{\mathcal{M},f}(V)$ is sub Hopf algebra of the primitively generated Hopf algebra $T(V)$, it is primitively generated by~\cite[Proposition 6.13]{MM} and so there is a natural epimorphism
$$
\phi_n\colon P_nB^{\mathcal{M},f}\twoheadrightarrow Q_nB^{\mathcal{M},f}.
$$
From Proposition~\ref{proposition2.10}(4), the map $\phi_n$ admits a natural cross-section because $Q_nB^{\mathcal{M},f}$ is $T_n$-projective. Thus there is a subfunctor $D^{\mathcal{M},f}_{n}$ of $P_nB^{\mathcal{M},f}$ such that
$$\phi_n|\colon D^{\mathcal{M},f}_n\to Q_nB^{\mathcal{M},f}$$ is a natural isomorphism. Since $D^{\mathcal{M},f}_n$ is $T_n$-projective,
$$
D^{\mathcal{M},f}_n\subseteq P_nB^{\mathcal{M},f}\cap L_n
$$
from the lines in the proof of Lemma~\ref{lemma5.1}. Thus $D^{\mathcal{M},f}_n$ is a $T_n$-projective subfunctor of $L_n$. From the fact that $D^{\mathcal{M},f}\cong Q_nB^{\mathcal{M},f}$ and $B^{\mathcal{M},f}$ is isomorphic to the tensor algebra generated by $Q_nB^{\mathcal{M},f}$ with $n\geq 1$, the inclusion
$$
\bigoplus_{n=1}^\infty D_n\rInto B^{\mathcal{M},f}
$$
induces a natural isomorphism
\begin{equation}\label{equation6.1}
T\left(\bigoplus_{n=1}^\infty D_n\right)\cong B^{\mathcal{M},f}.
\end{equation}
Since the algebra $B^{\mathcal{M},f}$ is generated by $L_{m_ip^r}(V)$ for $m_i\in\mathcal{M}$ and $0\leq r< f(i)$, we have
\begin{equation}\label{equation6.2}
D^{\mathcal{M},f}_n=0\textrm{ if } n\not=m_ip^r \textrm{ for  some } m_i\in \mathcal{M}\textrm{ and some } 0\leq r< f(i).
\end{equation}
Let $\{m_ip^r \ | \ m_i\in\mathcal{M}, 0\leq r< f(i)\}=\{n_1,n_2,\ldots\}$ with $n_1<n_2<\cdots$ and let $\alpha$ be the cardinality of the set $\{m_ip^r \ | \ m_i\in\mathcal{M}, 0\leq r< f(i)\}$. Then decomposition~(\ref{equation6.1}) becomes
\begin{equation}\label{equation6.3}
T\left(\bigoplus_{i=1}^\alpha D_{n_i}\right)\cong B^{\mathcal{M},f}.
\end{equation}
and so a natural isomorphism
\begin{equation}\label{equation6.4}
PT\left(\bigoplus_{i=1}^\alpha D_{n_i}\right)=L^{\res}\left(\bigoplus_{i=1}^\alpha D_{n_i}\right)\cong PB^{\mathcal{M},f}=B^{\mathcal{M},f}\cap L^{\res}.
\end{equation}
According to Proposition~\ref{proposition4.1}, the sub Hopf algebra $B^{\mathcal{M},f}$ is also a natural coalgebra retract of $T$ if we change the ground ring $R$ to be $p$-local integers. Notice that $PT=L$ and $PB^{\mathcal{M},f}=B^{\mathcal{M},f}\cap L$ when $R=\Z_{(p)}$. By changing the ground ring back to $\Z/p$ and then extending it to $\bfk$, we have the following decomposition:
\begin{equation}\label{equation6.5}
L\left(\bigoplus_{i=1}^\alpha D_{n_i}\right)\cong B^{\mathcal{M},f}\cap L.
\end{equation}

For applying the Hilton-Milnor Theorem to determine $B^{\mathcal{M},f}\cap L_n$, let us recall the terminology of \textit{basic product} from~\cite[p.512]{Whitehead}. Let $x_1,\ldots,x_k$ be letters. A monomial means a formal product $w=x_{i_1}x_{i_2}\cdots x_{i_n}$ with $1\leq i_1,\ldots,i_t\leq k$, where the word length $n$ is called the \textit{weight} of $w$. We define the \textit{basic products} of weight $n$, by induction on $n$; and for each such product, a non-negative integer $r(w)$, called its \textit{rank}. These are to be linearly ordered, in such a way that $w_1<w_2$ if the weight of $w_1$ is less than the weight of $w_2$. The \textit{serial number $s(w)$} is the number of basic products $\leq w$ in term of this ordering. The basic products of weight $1$ are the letters $x_1,\ldots,x_k$ with the order $x_1<x_2<\cdots<x_k$. The rank $r(x_i)=0$ and the serial number $s(x_i)=i$. Suppose that the basic products of weight less than $n$ have been defined and linearly ordered in such a way that $w_1<w_2$ if the weight of $w_1$ is less than that of $w_2$; and suppose that the rank $r(w)$ of such a product has been defined. Then the basic products of weight $n$ are all monomials of the $w_1w_2$ of weight $n$, for which $w_1$ and $w_2$ are basic products, $w_2<w_1$ and $r(w_1)\leq s(w_2)$. Give these an arbitrary linear order, and define $r(w_1w_2)=s(w_2)$.

Let $\mathcal{W}_k$ be the set of all basic products on the letters $x_1,\ldots,x_k$ by forgetting the ordering. Then
$$
\mathcal{W}_k\subseteq\mathcal{W}_{k+1}
$$
for each $k$. Let
$$
\mathcal{W}_{\infty}=\bigcup_{k=1}^\infty \mathcal{W}_k.
$$
The elements in $\mathcal{W}$ are called basic products on the sequence of the letters $x_i$ for $i\geq 1$. For each basic product $w=x_{i_1}\cdots x_{i_t}\in \mathcal{W}_{\alpha}$, define
\begin{equation}\label{equation6.6}
w(D^{\mathcal{M},f})=D^{\mathcal{M},f}_{n_{i_1}}\otimes\cdots\otimes D^{\mathcal{M},f}_{n_{i_t}}
\end{equation}
with the \textit{tensor length with respect to $D^{\mathcal{M},f}$}
$$
d(w)=n_{i_1}+n_{i_2}+\cdots+n_{i_t}
$$
and the natural transformation
$$
\phi_w\colon w(D^{\mathcal{M},f})(V)=\longrightarrow T\left(\bigoplus_{i=1}^\alpha D_{n_i}(V)\right)\cong B^{\mathcal{M},f}(V)
$$
given by
$$
\phi_w(z_1\otimes z_2\otimes\cdots\otimes z_t)=[[z_1,z_2,\ldots,z_t]
$$
for $z_j\in D^{\mathcal{M},f}_{n_{i_j}}(V)$. Then the map $\phi_w$ extends uniquely to a natural transformation of Hopf algebras
$$
T\phi_w\colon T(w(D^{\mathcal{M},f}))\longrightarrow T\left(\bigoplus_{i=1}^\alpha D_{n_i}(V)\right)\cong B^{\mathcal{M},f}(V)
$$
by the universal property of tensor algebras. Now by taking the homology from the Hilton-Milnor Theorem~\cite[Theorem 6.7]{Whitehead}, we have the natural isomorphism of coalgebras
\begin{equation}\label{equation6.7}
\theta\colon \bigotimes_{w}T(w(D^{\mathcal{M},f}))\rTo^{\cong} T\left(\bigoplus_{i=1}^\alpha D_{n_i}\right)\cong B^{\mathcal{M},f},
\end{equation}
where $w$ runs over all basic products in $\mathcal{W}_{\alpha}$, the tensor product is linearly ordered and the natural transformation $\theta$ is given by the ordered product of $T\phi_w$ which is well-defined because the tensor length $d(w)$ tends to $\infty$ as the weight of $w$ tends to $\infty$. By restricting to Lie powers, we have the decomposition
\begin{equation}\label{equation6.8}
\theta|\colon \bigotimes_{w}L(w(D^{\mathcal{M},f}))\rTo^{\cong} L\left(\bigoplus_{i=1}^\alpha D_{n_i}\right)\cong B^{\mathcal{M},f}\cap L.
\end{equation}
By taking tensor length, we obtain the following decomposition theorem.

\begin{thm}\label{theorem6.1}
Let $\mathcal{M}=\{m_i\}_{i\in I}$ be finite or infinite set of positive integers prime to $p$ with each $m_i>1$ and let $f\colon I\to \{0,1,2,\ldots\}\cup\{\infty\}$ be a function. Then there exists $T_{m_ip^r}$-projective subfunctor $D^{\mathcal{M},f}_{m_ip^r}$ of $L_{m_ip^r}$ for each $m_i\in\mathcal{M}$ and $0\leq r< f(i)$ such that
$$
L_{m_ip^r}=\bigoplus_{d(w)|m_ip^r}L_{m_ip^r/d(w)}(w(D^{\mathcal{M},f}))
$$
for $m_i\in\mathcal{M}$ and $0\leq r< f(i)$, where $w$ runs over basic products with $d(w)|m_ip^r$.\hfill $\Box$.
\end{thm}

\noindent\textbf{Note.} Since each $D_{n_i}$ is $T_{n_i}$-projective, the tensor product $w(D^{\mathcal{M},f})$ is $T_{d(w)}$-projective. If $m_ip^r/d(w)$ is prime to $p$, then the Lie power $L_{m_ip^r/d(w)}(w(D^{\mathcal{M},f}))$ is $T_{m_ip^r}$-projective. Thus the non $T_{m_ip^r}$-projective summands of $L_{m_ip^r}$ occurs in the those factors $L_{m_ip^r/d(w)}(w(D^{\mathcal{M},f}))$ with $m_ip^r/d(w)\equiv0\mod{p}$.

Another remark is that the multiplicity of each factor $L_{m_ip^r/d(w)}(w(D^{\mathcal{M},f}))$ can be determined as follows: Let $w$ be a basic product involving the letters $x_{j_1},\ldots,x_{j_k}$ such that $x_{j_i}$ occurs $l_i$ times and $d(w)|m_ip^r$. According to~\cite[(6.4), p.514]{Whitehead}, the multiplicity of the factor $L_{m_ip^r/d(w)}(w(D^{\mathcal{M},f}))$ is given by the formula
$$
\frac{1}{l}\sum_{d|l_0}\mu(d)\frac{\left(\frac{l}{d}\right)!}{\left(\frac{l_1}{d}\right)!\cdots \left(\frac{l_k}{d}\right)!},
$$
where $\mu$ is the M\"obius function, $l_0$ is the greatest common divisor of $l_1,\ldots,l_k$ and $l=l_1+\cdots +l_k$.

\begin{example}{\rm
Let $\bfk$ be of characteristic $2$. Let $\mathcal{M}=\{m_1=3\}$ and let $f(1)=3$. Then we have natural coalgebra-split sub Hopf
algebra
$$
B^{\mathcal{M},f}(V)=\la L_3(V), L_6(V),L_{12}(V),\ra
$$
of $T(V)$ with $D^{\mathcal{M},f}_3=L_3$, $D^{\mathcal{M},f}_6\cong L'_6=L_6/L_2(L_3)$ and
$$
D^{\mathcal{M},f}_{12}\cong L_{12}/\left([L'_6,L'_6]\oplus [[L'_6,L_3],L_3]\oplus
L_4(L_3)\right)
$$
From Theorem~\ref{theorem6.1}, we have the decomposition
$$
\begin{array}{rcl}
L_{12}&=&D_{12}\oplus L(D_3\oplus D_6)\cap L_{12}\\
&=&D_{12}\oplus L_4(D_3)\oplus L_2(D_6)\oplus [[D_6, D_3],D_3]\\
&\cong& D_{12}\oplus L_4(L_3)\oplus L_2(L'_6)\oplus [[L'_6,L_3],L_3].\\
\end{array}
$$
By comparing with~\cite[Theorem 10.7]{SW1} or~\cite[Theorem 4.4]{Bryant-Schocker}, the $T_{12}$-projective summand
$$[[L'_6,L_3],L_3]\cong L'_6\otimes L_3\otimes L_3$$ can be recognized in our decomposition for $L_{12}$. \hfill $\Box$ }
\end{example}

Let $\mathcal{M}$ be the set of all positive integers $m_i$ with $m_i$ prime to $p$ and $m_i>1$ and let $f(i)=\infty$ for all $i$. Then the set
$$
\{m_ip^r\ | \ m_i\in \mathcal{M}\ r\geq 0\}=\mathbb{N}\smallsetminus \{1,p,p^2,p^3,\ldots\}.
$$
Let
$$
\bar D_n=D^{\mathcal{M},f}_n
$$
for $n$ not a power of $p$. As a special case of Theorem~\ref{theorem6.1}, we have the following:

\begin{cor}\label{corollary6.3}
There exists $T_n$-projective subfunctor $\bar D_n$ of $L_n$ for each $n$ not a power of $p$ such that
$$
L_{m}=\bigoplus_{d(w)|m}L_{m/d(w)}(w(\bar D))
$$
for any $m$ not a power of $p$, where  $w$ runs over basic products with $d(w)|m$.\hfill $\Box$.
\end{cor}

\end{document}